\documentclass[11pt]{article}

\usepackage[a4paper,headinclude=false,footinclude=false,DIV=10]{typearea}
\usepackage{setspace}

\usepackage{amsmath,amsthm,amssymb,amscd}
\usepackage{tikz}
\usepackage{tikz-cd}

\usepackage{epsfig,graphicx}
\usepackage[colorlinks=false]{hyperref}
\usepackage[ngerman,english]{babel}

\usepackage[utf8]{inputenc}

\usepackage{fourier} 
\usepackage[scaled=0.875]{helvet} 

\newtheoremstyle{mystyle}{}{}{\rmfamily}%
{}{\normalfont\bfseries}{ }{ }{}

\newtheorem{theorem}{Theorem}[section]
\newtheorem{proposition}[theorem]{Proposition}
\newtheorem{lemma}[theorem]{Lemma}
\newtheorem{corollary}[theorem]{Corollary}
\newtheorem{defn}[theorem]{Definition}
\theoremstyle{mystyle}
\newtheorem{remark}[theorem]{Remark}

\theoremstyle{plain}
\newtheorem{maintheorem}{Theorem}

\newcommand{\R}{\mathbb{R}}
\newcommand{\N}{\mathbb{N}}

\newcommand{\Z}{\mathbb{Z}}

\newcommand{\cW}{\mathcal{W}}
\newcommand{\cL}{\mathcal{L}}

\newcommand{\cH}{\mathcal{H}}
\newcommand{\cHL}{\mathcal{H}_{\hbox{\tiny loc}}}

\newcommand{\te}{{\theta}}
\newcommand{\G}{\mathbb{G}}
\newcommand{\X}{\mathbb{X}}
\newcommand{\diam}{\mathrm{diam}}

\newcommand{\1}{\mathbf{1}}

\DeclareMathOperator{\id}{id}

\begin{document}
\selectlanguage{english}
\vspace*{1cm}

\begin{center}
\setstretch{2}
{\Huge \bfseries The Martin boundary of an extension by a hyperbolic group}\\
%
Sara Ruth Pires Bispo\textsuperscript{(1)} and Manuel Stadlbauer\textsuperscript{(2)}\\
\bigskip
{\scriptsize	
\textsuperscript{(1)}
Centro das Ciências Exatas e das Tecnologias, Universidade Federal do Oeste da Bahia,
47808-021 Barreiras  (BA), Brazil\\[-.2cm]
\textsuperscript{(2)}
Departamento de Matemática, Universidade Federal do Rio de Janeiro,
Ilha do Fundão,  21941-909 Rio de Janeiro (RJ), Brazil\\[-.2cm]
}
\bigskip \bigskip
{\small \today}
\end{center}

\begin{quote}
We prove uniform Ancona-Gou\"ezel-Lalley inequalities for an extension by a hyperbolic group $G$ of a Markov map which allows to deduce that the visual boundary of the group and the Martin boundary are Hölder equivalent. As application, we identify the set of minimal conformal measures of a regular cover of a convex-cocompact CAT(-1)-manifold with the visual boundary of the covering group, provided that this group is hyperbolic. 
\end{quote}

\section{Introduction and statement of main results}

A standard approach for the encoding of the behaviour at infinity of a transient random walk is to analyse the positive harmonic functions of the associated Markov operator. As observed by Martin in the context of partial differential equations (\cite{Martin:1941}), these harmonic functions also have a topological description through potential theory. Namely, the harmonic functions are related to possible limits of the Green kernel at infinity, and hence, to the boundary of the minimal compactification such that the Green kernel extends to a continuous function on this new domain. However, as the construction only requires a well-defined Green function, it is applicable in a wide range of situations, from elliptic equations  and transient Markov processes  to conformal densities on CAT(-1)-spaces (see, e.g., \cite{Murata:2002,Revuz:1984,Roblin:2011}).  However, due to the generality of the approach, the explicit description of this Martin boundary for specific situations is often non-trivial. To give two examples of those explicit characterizations in the context of simple random walks on discrete groups, the Martin boundary of $\Z^d$ is a singleton whereas the one of the free group coincides with its visual boundary (see, e.g., \cite{Woess:1994a}). 

The other fundamental objects of this article are metric spaces which are 
hyperbolic in the sense of Gromov. The abstract definition of these spaces requires a uniform estimate of the Gromov product  
(see Section \ref{sec:AG-inequalities}), which can be reduced to a simple geometric property if the space is a geodesic space. That is, a geodesic space is $\delta$-hyperbolic if each side of an arbitrary triangle is always contained in the  $\delta$-neighbourhoods of the other two. Furthermore, each Gromov hyperbolic space comes with its visual boundary, 
which abstractly is defined through the asymptotic behaviour of the Gromov product (see Section \ref{sec:geometry-Martin}), but also can be defined through geometric data coming from Busemann's function in case of a geodesic space (see \cite{CoornaertPapadopoulos:2001}). Due to this second characterisation and the fact that the boundary comes with a natural metric, the visual boundary today is the standard tool for encoding and analysing the behaviour of geodesics and horospherical foliations.

However, even though these boundary constructions differ completely as the first is based on abstract potential theory whereas the second uses geometric phenomena, it is known since the works of 
Ancona in \cite{Ancona:1987} on elliptic operators on negatively curved manifolds 
and Gouëzel-Lalley in \cite{GouezelLalley:2013,Gouezel:2014} for symmetric, simple random walks on hyperbolic groups that the Martin and the visual boundary coincide in these cases. Ancona's contribution, from a general viewpoint, is a 
strategy of proof, which allows to deduce a geometric description of the Martin boundary through minimal harmonic functions from the so called Ancona inequalities, whereas the work of Gouëzel and Lalley provides us with an argument which allows to obtain uniform Ancona inequalities and a relation between the Green kernel and the Gromov product (see Theorems \ref{mainthmA} and \ref{theo:Ancona-Gouezel inequalities} for the setting in here). 
A further geometrization related to simple random walks is due to Kaimanovich (\cite{Kaimanovich:2000}), who proved that the visual boundary of a Gromov hyperbolic group coincides with the Poisson boundary, that is the set of bounded harmonic functions. However, it is worth noting that, even though the statements are similar, the method of proof in \cite{Kaimanovich:2000} is based on a submultiplicative ergodic theorem (see, also, \cite{KarlssonMargulis:1999}) and therefore is applicable also to  random walks whose transitions neither have to be symmetric nor finitely supported (for transitions with exponential tails, see \cite{Gouezel:2015}). On the other hand, as the method of Gouëzel and Lalley allows to study $\rho$-harmonic functions, where $1/\rho \geqslant 1$ is the radius of convergence of the Green function, these seemingly weaker results (in the generality of the results in here) give rise to applications to regular covers of negatively curved manifolds at their intrinsic exponent of convergence as in Theorem \ref{maintheo:regular cover} below.        

\medskip
We now proceed with the setting and the statement of our main results. Throughout, we assume that  $\theta: (\Sigma,\mu) \to (\Sigma,\mu)$ is a probability preserving, topologically mixing and noninvertible Markov map of a standard probability space with respect to a finite partition  such that $\log d\mu\circ \theta /d\mu$ has a Hölder continuous representative (see Definitions \ref{def:Markov_map} and \ref{def:gm-map}). Furthermore, we throughout assume that $G$ is a discrete group and that $\kappa: \Sigma \to G$ is constant on the atoms of the Markov partition. The extension of $\theta$ by $G$, or group extension for short, is then defined by  
\[ T: \Sigma \times G \to  \Sigma \times G, (x,g) \to (\theta(x), g \kappa(x))), \]
and is the key object of this article. Observe that, as $\mu$ is $\theta$-invariant, the product of $\mu$ and the Haar measure on $G$ is $T$-invariant. In particular, the transfer operator associated to $T$ can be written as, for $f: \Sigma \times G \to \R$ in a suitable function space,
\begin{align*}
\cL(f)(x,g)  = \sum_{T(y,h) = (x,g)} \frac{d\mu}{d\mu\circ\theta}(y) \; f(y,h)
\end{align*} 
and satisfies $\cL(\mathbf{1}) = \mathbf{1}$. 
It is  worth noting here that group extensions are random walks with dependent increments by identifying 
$\mu\left(\left\{ x \in \Sigma: T^n(x,g) \in \Sigma \times \{h\}\right\}\right)$ with the probability of a transition from $g \in G$ to $h \in G$ in time $n$. Moreover, as each simple random walk can be identified with a group extension, 
group extensions of Markov maps with finite partitions generalise simple random walks with respect to a finitely supported transition rule.  

We now return to the general group extensions o Markov maps. In this setting, the Green function is no longer a function but an operator defined by  
\[ \G_r(f):=  \sum_{n=0}^\infty r^n \mathcal{L}^n(f),\]
where  $f$ is an element of a suitable function space (see Proposition \ref{prop:action of G_r}) and $r < 1/\rho$, where $1/\rho$ is the radius of convergence of $r \mapsto \G_r(\mathbf{1}_{\Sigma \times \{\id\}})(x,\id)$. Furthermore, observe that, by general ergodic theory, the operator $\G_r$ can be extended to $r = 1/\rho$, provided that the map $T$ is totally dissipative. 

We are now in position to specify the main objective of this article. That is, we are interested in relating the Martin boundary, that is the possible limits of $ \G_r(f)(z_n)/\G_r(\mathbf{1}_{\Sigma \times \{\id\}})(z_n)$ 
with the visual boundary of $G$, where $G$ is Gromov hyperbolic and $1 \leq r \leq 1/r$. 
As a first step, we prove a uniform Ancona-Gou\"ezel-Lalley inequality. That is, by combining the first part of Theorem \ref{theo:Ancona-Gouezel inequalities} with Remark \ref{rem:hyperbolic-amenable} one obtains the first main result which  states that the Green operator is almost multiplicative along geodesics. For ease of exposition, the statement here is formulated in the presence of geodesics.

\begin{maintheorem} \label{mainthmA}
Assume that $G$ is a non-elementary and word hyperbolic group, that $T$ is a topologically transitive and that $\G_R(\mathbf{1}_{\Sigma \times\{g \} })(\,\cdot\,,\id) \asymp \G_R(\mathbf{1}_{\Sigma \times\{\id \} })(\,\cdot\,,g)$. Then 
for any $D>0$ and $g,z,h \in G$ such that the distance between $z$ and  the geodesic segment $[g,h]$ is smaller than $D$,
and any $r \in [1,1/\rho]$,
\[ \G_r(\mathbf{1}_{\Sigma \times\{h\} })(\,\cdot\,,g) \asymp \G_r(\mathbf{1}_{\Sigma \times\{h\}}) (\,\cdot\,, z) \;  \G_r(\mathbf{1}_{\Sigma \times\{z\}}) (\,\cdot\,,g).  \]
\end{maintheorem}

The proof of this theorem makes use of the strategy of Gou\"ezel and Lalley in \cite{GouezelLalley:2013,Gouezel:2014} adapted to the setting of group extensions, which required to develop a potential theory for conformal and exzessive measures in order to perform the necessary substitution of subharmonic functions by exzessive measures (cf. Section \ref{sec:appendix}). Moreover, again by following \cite{GouezelLalley:2013,Gouezel:2014}, it is possible to obtain an estimate for the fluctuations of the Green operator through the Gromov product (see the second part of Theorem \ref{theo:Ancona-Gouezel inequalities}).   

The multiplicative estimate in Theorem \ref{mainthmA} is probably the key result in here, as   
it allows to employ Ancona's argument in order to obtain a geometric characterisation of the Martin boundary and the arguments in \cite{GouezelLalley:2013} to prove Hölder continuity of the Green kernel. 
The Martin boundary associated to general  transient Markov shifts was introduced by Shwartz in \cite{Shwartz:2019a} and is similar to the well-known construction based on Markov operators from probability theory, even though the building blocks of the  boundary are $\sigma$-finite conformal measures instead of positive harmonic functions. In the context of group extensions, the Martin boundary is defined by
\begin{align*} \mathcal{M}_r &:= \left. \left\{ (x,g) \in \Sigma \times G : T^n(x,g) \to \infty , \; \lim_{n \to \infty}   \frac{\G_r(\mathbf{1}_{\Sigma \times\{h\} })(T^n(x,g))}{\G_r(\mathbf{1}_{\Sigma \times\{\id \} })(T^n(x,g))
}  \hbox{ exists for all } h \in G\right\}\right/ \sim 
\end{align*}
where $T^n(x,g) \to \infty$ means that $(T^n(x,g))$ leaves any compact set and $(x,g) \sim (\tilde x,\tilde g)$ that the limits in the definition coincide for all $h \in G$. We also remark that our definition differs slightly from the ones by Shwartz in \cite{Shwartz:2019a,Shwartz:2019} where the defining class of functions is a dense subset of the set of Hölder functions with compact support instead of  $\{ \mathbf{1}_{\Sigma \times\{h\}} : h \in G\}$ as in our definition. We also would like to point our that similar results to Theorems \ref{mainthmA} and \ref{theo:B} recently and independently were obtained by Shwartz in \cite{Shwartz:2019} for the more general setting of locally finite shifts which carry the structure of a hyperbolic graph. However, the method in there does not allow to include the case $r = 1/\rho$. 

Our second principal result characterises $\mathcal{M}_r $ geometrically and reveals that the local influence of $\Sigma$ vanishes as $T^n(x,g) \to \infty$. 
\begin{maintheorem} \label{theo:B} Under the assumptions of Theorem \ref{mainthmA} and for any $1 \leq r \leq 1/\rho$, the following holds.  
  For each sequence $(\xi_n)$ in $G$ converging to $\sigma$ in the visual boundary $\partial G$, the limit 
\[  
 \mu_{\sigma}(f) := 
\lim_{n \to \infty} 
\frac{\G_r(f)(x,\xi_n)}{\G_r(\mathbf{1}_{\Sigma \times\{\id\} })(x,\xi_n)} 
\] 
exists for each Hölder function $f$ with compact support, the limit only depends on $\sigma$ and $\mu_{\sigma}$ is a minimal conformal measure. 
Furthermore, the map $\sigma \to \mu_\sigma$ is a bijection from  $\partial G$ to the set of minimal conformal measures. 
\end{maintheorem}

In fact, we prove more. Theorem \ref{theo:geometric-boundary} also gives important application of the above identification like a representation of any conformal measure as a convex combination of minimal ones and the exponential decay of $\G_r(\mathbf{1}_{\Sigma \times\{g\} })$, that is 
\[ \limsup_{n \to \infty} \max_{y \in \Sigma, |\gamma|=n} \sqrt[n]{\G_r(\mathbf{1}_{\Sigma \times\{g\} })\left(y, \gamma \right)} < 1.\]  
Furthermore, it also follows easily Theorem \ref{theo:geometric-boundary} that there is a bijection from $\partial G$ to $\mathcal{M}_r$. However, by a refinement of  Theorem \ref{mainthmA}, one obtains a Hölder continuous version of this statement (see Theorem \ref{theo:continuidade}).

\begin{maintheorem} Under the assumptions of Theorem \ref{mainthmA} and for any $1 \leq r \leq 1/\rho$, the map $\partial G \to \mathcal{M}_r$ induced by $\sigma \to \mu_\sigma$ is a Hölder continuous bijection with Hölder continuous inverse with respect to a different exponent.  
\end{maintheorem}

The above results have the following, canonical application to regular covers of convex-cocompact geodesic spaces through the coding construction in \cite{ConstantineLafontThompson:2020}. We recall the definition of a regular cover in our setting.  
Assume that $X$ is a CAT(-1)-space (see, e.g., \cite{DasSimmonsUrbanski:2017}) and that $\Gamma$ is a discrete subgroup of the isometries of $X$ which acts convex-cocompactly on $X$. Then, as it is well-known, $X/G$ is a local CAT(-1)-space with compact convex core. Now assume that $Y$ is a cover of $X/G$. We then refer to $Y$ as a regular cover if there exists a normal subgroup $N$ of $\Gamma$ such that $X/N$ and $Y$ are isometric. In this setting, the above provides a complete description of the space of $\delta(N)$-conformal measures for $N$ (for details and the proof, see Theorem \ref{theo:geometric-application}).  The following theorem both complements and extends the recent result by Shwartz in \cite{Shwartz:2019} for cocompact Fuchsian groups and $s$-conformal measures with $s > \delta(N)$.  

\begin{maintheorem} \label{maintheo:regular cover} Assume that $N$ is non-elementary, that $G:=\Gamma/N$ is word hyperbolic and that the geodesic flow associated with $X/N$ is topologically transitive.  Then the set of minimal, $\delta(N)$-conformal measures and $\partial G$ coincide.       
\end{maintheorem}

This result might be seen as a further contribution to a list of analogies between the ergodic behaviour of  
the geodesic flow on regular covers and random walks on the covering group. Namely, even though there does not exist a complete dictionary, there are several parallel results, like Rees' version of Polya's result on the transience of the simple random walk for  $\Z^d$-covers (\cite{Rees:1981}), or Brooks' amenability criterium (\cite{Brooks:1985}) in the sprit of Kesten (\cite{Kesten:1959a}).

\section{Group extensions of Markov maps}
Recall that a Markov map (or Markov fibred systems) is defined as follows  (see, e.g.  \cite{AaronsonDenkerUrbanski:1993,Aaronson:1997}).
\begin{defn}\label{def:Markov_map} Suppose that $(\Omega,\mathcal{B},\mu)$ is a standard probability space and $\alpha$ is an at most countable partition of $\Omega$ into measurable sets of strictly positive measure. We refer to $(\Omega,\theta,\mu,\alpha)$ as a {Markov map} if, for all $a,b \in \alpha$,
\begin{enumerate}
\item  $\theta|_a : a \to \theta(a)$ is invertible, bimeasurable and non-singular,
\item  either $\mu(a \cap \theta(b)) = 0$ or $\mu(a \cap (\theta(b))^\mathbf{c})=0$,
\item and, for $\alpha_n:= \left\{ a_1 \cap \theta^{-1}a_2 \cdots \cap \theta^{n-1}a_{n} : a_i \in \alpha,  i=1, \ldots, n \right\}$, the $\sigma$-algebra generated by $\{\alpha_n:n>0\}$ is equal to $\mathcal{B}$ up to sets of measure $0$.
\end{enumerate}
\end{defn}
Each Markov fibred system is a factor of a topological Markov chain $(\Sigma,\theta)$ where $\theta$ is the left shift $\sigma$ acting on
\[ \Sigma := \left\{ (a_i: i \in \N): a_i \in \alpha \hbox{ and } \mu(a_{i+1}) \cap \theta(a_i))> 0 \; \forall i = 1,2, \ldots \,\right\}\]
which can be deduced from the following.
Set $\cW^1 := \alpha$, and for $w_i \in \cW^1$ ($i=1,\ldots,n$) we say that  $w = (w_1 \ldots w_n)$ is an \emph{admissible word of length} $n$ if $\theta(w_i) \supset w_{i+1}$ for $i=1, \ldots, n-1$.
The {set of admissible words of length $n$} will be denoted by $\cW^n$, the length of $w \in \cW^n$ by $|w|$ and the set of all admissible words by $\cW^\infty = \bigcup_n \cW^n$.
Then
\begin{equation}\label{eq:correspondence_cylinders_words} \cW^n \to \alpha_n,\quad  (w_1 \ldots w_n) \mapsto [w_1 \ldots w_n] := \bigcap_{k=1}^{n} \theta^{-k+1}(w_k) \end{equation}
defines a bijection between $\cW^n$ and $\alpha_n$. Moreover, \eqref{eq:correspondence_cylinders_words} combined with \emph{(iii)} in Definition \ref{def:Markov_map} allows to lift $\mu$ to a probability measure $\mu^\ast$ on $\Sigma$ such that the limit of \eqref{eq:correspondence_cylinders_words} as $n$ tends to infinity defines a measure theoretical isomorphism between $(\Sigma,\mu^\ast)$ and  $(\Omega,\mu)$ which extends to a conjugation of $(\Sigma,\mu^\ast,\sigma)$ and  $(\Omega,\mu,\theta)$. Therefore, by abuse of notation, we identify both systems with
$(\Sigma,\mu,\theta,\alpha)$.

An important consequence of this identification is that it allows 
to effectively describe the preimage structure and induces a topology on $\Sigma$ such that $\theta$ is uniformly expanding.
That is, each $w \in \cW^n$ can be identified with an inverse branch of $\theta^n$ as follows. Since  $\theta^n$ maps ${[w]}$ injectively onto its image, its inverse $\tau_w: \theta^n([w]) \to [w]$ is well defined and by \emph{(i)} in Definition \ref{def:Markov_map},
\[0 <  \varphi_w(x) := \frac{d \mu \circ \tau_w}{d \mu}(x) < \infty \]
for $\mu$-a.e. $x \in \theta^n([w])$. Furthermore, $\Sigma$ comes with a canonical topology generated by $\{[w] : w \in \cW^\infty\}$ which coincides with the topology induced by the metric  $d_r$ defined by, for any $r \in (0,1)$,
 \[ d_r((x_i),(y_i)) := r^{\min\{i: x_i \neq y_i\}}.\]
The topology allows to define topological transitivity and mixing as follows. We refer to $(\Sigma,\theta)$ as \emph{topologically transitive} if for all $a,b \in \alpha$, there exists $n_{a,b}\in \N$ such that  $\mu(\te^{n_{a,b}}(a)\cap b)>0$ and
 as \emph{topologically mixing} if for all $a,b \in \alpha$, there exists $N_{a,b}\in \N$ such that $\mu(\te^{n}(a)\cap b)>0$ for all $n \geq N_{a,b}$.

\begin{defn}\label{def:gm-map} We refer to the Markov map $(\Sigma,\theta,\mu,\alpha)$ as a Gibbs-Markov map  of finite type
if $\alpha$ is finite, $(\Sigma,\theta)$ is topologically mixing and
there exists $C>0 $, $r \in (0,1)$ such that, for all $w \in \cW^\infty$ and a.e. $x,y \in X$,
\[ \left|\log{\varphi_w(x)} - \log{\varphi_w(y)}\right| \leq C d_r(x,y).\]
\end{defn}

The key feature of a the Gibbs-Markov property stems from the fact that the transfer operator $\cL_\theta: L^1(\mu) \to L^1(\mu)$ of $\theta$, which is defined as the dual of $U_\theta: L^\infty(\mu) \to L^\infty(\mu) $, $f \mapsto f\circ \theta$ and can be written as
\[ \cL_\theta(f) := \sum_{w \in \cW} \varphi_w  f\circ \tau_w,  \]
acts with a spectral gap on the space of Hölder continuous functions. This implies that there exists a unique invariant probability $m$ absolutely continuous with respect to $\mu$ and, in particular, that $\log dm/d\mu$ is Hölder continuous and $(\Sigma,\theta,m,\alpha)$ also has the Gibbs-Markov property (see  \cite{AaronsonDenkerUrbanski:1993,Aaronson:1997,Sarig:2003a}).

Now suppose that $G$ is a discrete group and that $\kappa: \Sigma \to G$ is a map such that $\kappa$ is measurable with respect to $\alpha$. We then refer to the skew product
\[ T: \Sigma \times G \to  \Sigma \times G, (x,g) \to (\theta(x), g \kappa(x))) \]
as a group extension. Furthermore, observe that
$\kappa^n: \Sigma \to G$, $x  \mapsto \kappa(x) \cdot \kappa(\theta(x)) \cdots \kappa(\theta^{n-1}(x)) $
is measurable with respect to $\alpha_n$. Therefore, for $w \in \cW^n$, we define $\kappa_w$ as $\kappa_w:= \kappa^n(x)$ for some $x \in [w]$. Moreover, by a slight abuse of notation, let $\tau_w$ also refer to the
inverse branch of $T^n$ on $[w,g]$. That is $\tau_w(x,g) := (\tau_w(x), g \kappa_w^{-1})$, whenever $\tau_w(x)$ is well defined (i. e., $x \in \theta^n([x])$). Moreover, for $f : \Sigma \times G \to \R$, set   $f_w(x,g) := f \circ \tau_w(x,g)$ for $x \in \theta^n([x])$ and $f_w(x,g) =0$ otherwise.
 The iterates of the transfer operator associated with the group extensions now can be written in short form as
\begin{align*}
\cL^n(f)(x,g) & := \sum_{v \in \cW^n} \varphi_v(x) f_v(x,g).
\end{align*}
The following Lemma provides an important estimate for the distortion of the iterates of $\cL$.
\begin{lemma} \label{lemma:distortion of L}
Suppose that $T$ is a topologically transitive extension of a Gibbs-Markov map $\theta$ of finite type. Then, for each $h \in G$, there exist $K_h > 0$ and $N_h \in \N$ with the following property.
For all  $x,\tilde{x} \in \Sigma$,  $g, \tilde{g} \in G$ with $h = g^{-1}\tilde{g}$, $m \in \N$, $L \geq 0$ and
$f:\Sigma \times G \to [0,\infty)$ such that,  for $z,\tilde{z}$ in the same cylinder of length $m$, either $f(z)=f(\tilde{z})=0$ or $\left|  \textstyle  {f(z)}/{f(\tilde{z})} -1 \right|\leq L
$, there
exists $k \leq N_h$ such that for all $n \geq m$
\[ \cL^n(f)(x,g) \leq  \left( K_{\tilde{g}^{-1}g} (L + 1) \right) \cL^{n+k}(f)(\tilde{x},\tilde{g}) . \]
\end{lemma}
\begin{proof}
If $g = \tilde{g} $ and  $x,\tilde{x}$ are in the same cylinder and $n \geq m$, then
\begin{align*}
& \left| \cL^n(f)(x,g) - \cL^n(f)(\tilde{x},\tilde{g}) \right|
\\
 \leq &
  \sum_{v\in \cW^n}  \left|\left(  \varphi_{v}(x) - \varphi_{v}(\tilde{x}) \right) f_v(x,g)  \right| +
  \sum_{v\in \cW^n}   \varphi_{v}(\tilde{x})  \left| \left( f_v(x,g)  -  f_v(\tilde{x},g) \right)  \right|
  \\
\leq & C_\varphi  d_r(x,\tilde{x}) \sum_{v\in \cW^n}  \varphi_{v}(x) |f_v(x,g)| +
C_\varphi \sum_{v\in \cW^n}  \varphi_{v}(x) |f_v(x,g)| \cdot \left| 1 - \textstyle  \frac{f_v(\tilde{x},g)}{f_v({x},g)}   \right|\\
\leq &  C_\varphi \left( d_r(x,\tilde{x}) + \sup_{v \in \cW^n} \left| 1 - \textstyle  \frac{f_v(\tilde{x},g)}{f_v({x},g)}   \right| \right)  \cL^n(f)(x,g) \leq C_\varphi(1+ L)  \cL^n(f)(x,g).
\end{align*}
Now assume that  $x,\tilde{x}$ are not in the same cylinder. Then, by transitivity, there exists $y$ in the same cylinder as $x$ and $k \in \N$ such that $T^k(y,g) = (\tilde{x},g)$, or, equivalently, there exists $v \in \cW^k$ with $(y,g) = \tau_v(\tilde{x},g)$. Hence,  by the above,
\begin{align*}
\cL^{n+k}(f)(\tilde{x},g) \geq \varphi_v(\tilde{x}) \cL^{n}(f)(y,g) \geq \frac{ \varphi_v(\tilde{x})}{1 + C_\varphi(L+1)}\cL^{n}(f)(x,g).
\end{align*}
Moreover, as $\theta$ is of finite type, $v \in \cW^k$ can be chosen within a finite set, which proves the assertion  for  $(x,g)$ and $(\tilde{x},g)$ with respect to $K_{\id} + K_{\id}L$, for some $K_{\id}$ sufficiently large. The proof of the general case is almost the same: For $(\tilde{x},\tilde{g})$, there exists by transitivity $w \in \cW^k$ such that $g\kappa^k(w) = \tilde{g}$ and $\tilde{x} \in \theta^k([w])$. As $\theta$ is of finite type and $\kappa_w = g^{-1}\tilde{g}$, $w$ again can be chosen from a finite set. Moreover,
$\tau_w(\tilde{x},\tilde{g}) \in \Sigma \times \{g\}$. Hence, by the above, for some $l \leq N_{\id}$ and any $n \geq m$,
\[
\cL^{n+k+l}(f)(\tilde{x},\tilde{g}) \geq \varphi_w(\tilde{x}) \cL^{n+l}(f)(\tau_w(\tilde{x}),g) \geq \frac{ \varphi_w(\tilde{x})}{K_{\id} + K_{\id} L}\cL^{n}(f)(x,g).
\]
The assertion of the Lemma then follows from this.
\end{proof}

\section{The Green operator}
In analogy to the Green functions  in the theory of random walks, we formally define a family of operators by \[\G_r := \sum_{n=0}^\infty r^{n} \cL^n.\] In order to be able to specify invariant function spaces, first observe that by invariance of $\mu$, we have that $\cL(\1)=\1$ and hence, $\G_r(\1) = (1-r)^{-1}\1$ for $r \in [0,1)$. However, $\G_r$ might act on functions with compact support for some $r \geq 1$. In order to determine the critical value for this action, set $R:= 1/\rho$, where
\[\rho:=\limsup_{n}\sqrt[n]{\cL^{n}(\X_{g})(x,h)} \, \hbox{ with } \, \X_{g}:= \mathbf{1}_{\Sigma\times \{g\}}.\]
Observe that, by transitivity, $\rho$ does not depend on $g,h \in G$ and $x \in \Sigma$. In particular,  $\G_r(\X_g)$ is a finite function for each $g \in G$ and $0 \leq  r < R$ by Hadamard's formula for the radius of convergence of a power series. Hence, if $\rho <1$ (which holds if $G$ is nonamenable, see \cite{Stadlbauer:2013}), we have to consider values of $r$ bigger than $1$.

The main idea behind the construction of an invariant space is to consider functions, whose local Hölder coefficients are non-constant and dominated by positive eigenfunctions of $\cL$. That is, for $\alpha >0$ and $f:\Sigma \times G \to \R$, we define a function $D_\alpha(f): X \to [0, \infty]$ which is constant on cylinders of length 1 by
\[
D_\alpha(f)(z,g) := \sup_{x,y \in [a,g]} \frac{\left|f(x)-f(y)\right|}{d(x,y)^{\alpha}} \quad  \hbox{ for all } z \in [a].
\]
Furthermore, we refer to $E_\rho$ as the set of all positive, H\"older continuous $\rho$-subharmonic functions, that is $E_\rho := \left\{ h :  \|D_\alpha(h)\|_\infty < \infty, h> 0,     \cL(h) \leq \rho h \right\}$, which is non-empty for $\alpha$ equal to the Hölder exponent of $\log \varphi$ if $|\cW|< \infty$ by the main result in \cite{Stadlbauer:2019}. We are now in position to define the following spaces of Hölder continuous and locally Hölder continuous functions.
\begin{align*}
\cH_\alpha & :=\left\{ f:X \to \R \Big|  \|f\|_\infty < \infty,  \|D_{\alpha}(f)\|_\infty < \infty  \right\}\\
\cHL  & :=\left\{ f:X \to \R \Big| \exists  h \in E_\rho \hbox{ s.t. } \left| f \right| \leq h,  D_{\alpha}(f) \leq h \right\}
\end{align*}

\begin{proposition} \label{prop:action of G_r}
 Suppose that $\|D_\alpha(\log \varphi) \|_\infty < \infty$ and that $\cL_\theta(1)=1$. Then  $\G_r$ acts on $\cH_\alpha$ as a bounded operator with respect to $\|\cdot \|_\infty + \|D_\alpha(\cdot) \|_\infty$ for each  $r \in [0,1)$.  For $r \in [0,R)$,
$\G_r$ acts on $\cHL$  and
 there exists a constant $C>0$ such that, for all $f \in\cHL$ and $h \in E_\rho$ with $|f|\leq h$ and $ D_{\alpha}(f)\leq h$,
\[ |\G_{r}(f)| \leq \frac{R}{R-r} h, \quad D_{\alpha}(\G_{r}(f)) \leq \frac{C R}{R-r} h . \]
\end{proposition}

\begin{proof} For $r< 1$ and $f \in \cH_\alpha$, we have that \[\|\G_{r}(f)\|_\infty \leq \|f\|_\infty \|\G_{r}(1)\|_\infty  = \|f\|_\infty  \sum_n r^n = \|f\|_\infty (1-r)^{-1}.\]
It remains to show for the first part that $\|D_{\alpha}(\G_{r}(f))\|_\infty < \infty$. In order to do so, assume that $x,y \in [v]$ for some $v \in \cW^1$ and recall that, by the Gibbs-Markov property, there  exists $C_\varphi$, independent of $x,y$ and $w\in \cW^{n}$ such that $|1 - \varphi_w(x)/\varphi_w(y)| \leq C_\varphi d(x,y)^\alpha$ and $\varphi_w(y)/\varphi_w(x) \leq C_\varphi$. Hence,
\begin{align*}
&  \left|\G_{r}(f)(x,g) - \G_{r}(f)(y,g) \right|/d(x,y)^\alpha\\
& \leq  \frac{1}{d(x,y)^\alpha} \sum_{n \in \N_0, w\in \cW^{n}}  r^n \left| \varphi_w(x) -  \varphi_w(y)\right| |f_w(x,g)| +  \varphi_w(y)  \left| f_w(x,g)  - f_w(y,g)\right| \\
& \leq \sum_{n \in \N_0, w\in \cW^{n}} r^n C_\varphi   \varphi_w(y)\|f\|_\infty  +  \frac{r^n}{2^{\alpha n}} \varphi_w(y) \|D_{\alpha}(f)\|_\infty     = \frac{C_\varphi}{1-r} \|f\|_\infty  +  \frac{1}{1-\frac{r}{2\alpha}} \|D_{\alpha}(f)\|_\infty.
\end{align*}
Suppose that $r< R$ and $f\in\cHL $. Then there is
$h \in E_\rho$ with $|f|\leq h$ and $ D_{\alpha}(f)\leq h$, and
\begin{align*}
|\G_{r}(f)| \leq \G_{r}(|f|) \leq \G_{r}(h) = \sum_{n=0}^\infty r^{n} \cL^n(h) = \sum_{n=0}^\infty (\rho r)^{n} h =  \frac{R}{R-r} h.\end{align*}
Hence, it remains to show that $D_{\alpha}(\G_{r}(f))\ll h$. By similar arguments, for $x,y \in [v]$ for some $v \in \cW^1$,
\begin{align*}
 &  \left|\G_{r}(f)(x,g) - \G_{r}(f)(y,g) \right|/d(x,y)^\alpha\\
& \leq  \frac{1}{d(x,y)^\alpha} \sum_{n \in \N_0, w\in \cW^{n}}  r^n \left| \varphi_w(x) -  \varphi_w(y)\right| |f_w(x,g)| +  \varphi_w(y)  \left| f_w(x,g)  - f_w(y,g)\right| \\
& \leq \sum_{n \in \N_0, w\in \cW^{n}} r^n C_\varphi   \varphi_w(y) |f_w(x,g)| +  \frac{r^n}{2^{\alpha n}} \varphi_w(y) D_{\alpha}(f)(\tau_w(x)g\kappa_w^{-1})   \\
& \leq C_\varphi \left( \G_{r}(|f|)(x,g) + \G_{2^{-\alpha} r}(D_{\alpha}(f))(x,g) \right) \leq C_\varphi \left( \frac{R}{R-r} + \frac{2^\alpha R}{2^\alpha R-r} \right) h(x,g).
\end{align*}
\end{proof}

In order to extend the action further to $r=R$, we introduce the following notion of transience in analogy to the theory of random walks. Now suppose that $\mu$ is a non-singular measure on $X$ with respect to $T$. We then
 refer to $T$ as \emph{transient} or \emph{$\rho$-transient} if  $\G_R(\X_{\id})(x,\id) <\infty$ for all $x \in \Sigma$.

\begin{proposition} \label{prop:finiteness of G_R} Assume that $T$ is a topologically transitive, $\rho$-transient extension of a Gibbs-Markov map of finite type. Then, for $f \in \cHL$ and $A \subset G$ finite, we have that 
 $\G_R(\X_A) \in E_\rho$, $\G_R(\X_A \cdot f) \in \cHL$ and
\begin{align*} |\G_R(\X_A \cdot f)| & \leq  \|\X_A \cdot f \|_\infty \cdot  \G_R(\X_A), \\
D_\alpha(\G_R(\X_A \cdot f)) & \leq C_\varphi \|\X_A \cdot f \|_\infty   \G_R(\X_A)  +  \|D_\alpha(\X_A \cdot f) \|_\infty  \G_{2^{-\alpha}R}(\X_A).
\end{align*}
\end{proposition}

\begin{proof}
It follows from Lemma \ref{lemma:distortion of L} that
$\G_R(\X_{\id})(x,g) < \infty$ for all $(x,g) \in G$.
 As  $\G_R(\X_{\id})(x,g)$ is left invariant under multiplication by elements of $G$, it follows that  $\G_R(\X_g)(x,h)$ is finite for all $h\in G$. Hence, $\G_R(\X_A)(x,h) < \infty$ and, as it easily can be verified, $\G_R(\X_A) \in E_\rho$. The remaining assertions follow as in the proof above.    \end{proof}

\begin{remark}\label{rem:amenable}
Recall that a group is non-amenable if a strong isoperimetric inequality holds, that is
\[ \inf \left\{ \frac{|gA \triangle A| }{|A|} :  {A \subset G, |A| < \infty }  \right\}  > 0.\]
Moreover, non-amenability implies that $\rho < 1$ by \cite{Stadlbauer:2013}
and  $T: \Sigma \times G \to \Sigma \times G$ can not be conservative and ergodic for any measure by a result of Zimmer (\cite{Zimmer:1978a}) which implies that $T$ is transient (see, e.g., \cite[Prop. 2]{Stadlbauer:2019}). Hence, if $G$ is non-amenable, then  $R> 1$  in Proposition \ref{prop:action of G_r} and the assertions of Proposition \ref{prop:finiteness of G_R} hold. 
\end{remark}

As a consequence of Lemma \ref{lemma:distortion of L}, one immediately obtains the following independence of $  \G_r(f)(x,g)$ from $x$.
\begin{lemma} \label{lem:bounded distortion for the extended operator}
Assume that $T$ is a topologically transitive, $\rho$-transient
 extension of a Gibbs-Markov map of finite type and that there are $m \in \N$, $L \geq 0$ and $f:\Sigma \times G \to [0,\infty)$ such that,  for $z,\tilde{z}$ in the same cylinder of length $m$, either $f(z)=f(\tilde{z})=0$ or $\left|  \textstyle  {f(z)}/{f(\tilde{z})} -1 \right|\leq L
$. Then, for $1\leq r \leq R$ and $f$ such that   $\G_r(f)(x,\id) < \infty $,
\[ \left(K_{g^{-1}h}(L+1)\right)^{-1}   \G_r(f)(y,h) \leq \G_r(f)(x,g) \leq K_{h^{-1}g}(L+1)   \G_r(f)(y,h).\]
\end{lemma}

\begin{proof} It follows from Lemma \ref{lemma:distortion of L} that, from some $k \in \N$,
\begin{align*}
\G_r(f)(x,g)  & = \sum_{n=0}^\infty  r^n   \cL^{n}(f) (x,g) \leq  K_{h^{-1}g}(L+1) \sum_{n=0}^\infty  r^n   \cL^{n+k}(f) (y,h) \\
& = \frac{K_{h^{-1}g}(L+1) }{r^k} \sum_{n=k}^\infty  r^n   \cL^{n}(f) (y,h) \leq  {K_{h^{-1}g}(L+1) } \G_r(f)(y,h) .
\end{align*}
The second part follows from interchanging the roles of $(x,g)$ and $(y,h)$.
\end{proof}

\section{Ancona-Gouëzel inequalities for extensions by word hyperbolic groups}\label{sec:AG-inequalities}

Hyperbolic groups were introduced by Gromov (\cite{Gromov:1987}) in order to unify the theory of groups with a certain notion of negative curvature. In here, we exclusively consider the word metric on $G$ which we recall now. 
For a fixed, finite set $\mathfrak{g}$ of generators for $G$, the word metric is defined by
\[d(g,h) = \min\{k : ga_1 \ldots a_k=h  \hbox{ or }  ha_1 \ldots a_k=g \, a_i \in \mathfrak{g} \}.\]
In general, a metric space $(G,d)$ is a referred to as {Gromov hyperbolic} or $\delta$-hyperbolic in the sense of Gromov if $(G,d)$ is a geodesic space and there exists $\delta > 0$
 such that
\[ (x \cdot z)_\mathbf{o} \geq \min\left\{(x \cdot y)_\mathbf{o},(y \cdot z)_\mathbf{o} \right\} -\delta,  \]
for all $x,y,z,\mathbf{o} \in G$. In here,
 $(x \cdot y) = (x \cdot y)_\mathbf{o} $ refers to the Gromov product defined by
\begin{equation}\label{eq:gromov product} (x \cdot y)_\mathbf{o} := \frac{1}{2} ( d(x,\mathbf{o} ) + d(y,\mathbf{o} ) - d(x,y)  ). \end{equation}
In this situation, as $d$ is the word metric, one refers to $G$ as word hyperbolic. 
Important features of Gromov hyperbolic spaces are that triangles are $4\delta$-thin and that the Cayley graph of $G$ can be uniformly approximated by trees in the following sense (see Theorem 12 in \cite{GhysHarpe:1990}).
\begin{lemma}\label{lem:tree_approximation} Let $(M,d)$ be $\delta$-hyperbolic and $F \subset M$ with $|F| \leq 2^k +2$ and $\mathbf{o} \in M$. Then there exists a finite, rooted tree $T$ and $\Theta: F \to T$ such that
\begin{enumerate}
\item $d(x, \mathbf{o}) = d(\Theta(x), \Theta(\mathbf{o}))$ for all $x \in F$,
\item $d(x, y) - 2k \delta \leq d(\Theta(x), \Theta(y)) \leq d(x, y)$ for all $x,y \in F$.
\end{enumerate}
\end{lemma}

\begin{remark}\label{rem:hyperbolic-amenable}
A further important property of hyperbolic groups is related to non-amenability (see Rem. \ref{rem:amenable}). First recall that a group $G$ is elementary if $G$ has a cyclic subgroup of finite index. Then, a world hyperbolic group is either elementary (and therefore amenable) or non-elementary and non-amenable. This result is well known and can be deduced e.g. by combining Theorem A in \cite{CapraceCornulierMonod:2015} with the observation that the distance between two elements in $G$ is uniformly bounded from below. Hence, Remark \ref{rem:amenable} implies that for any non-elementary, word hyperbolic group $G$, the group extension $T$ is transient and $R> 1$.   
\end{remark}

For extensions by hyperbolic groups, we now prove a strong version of Ancona's inequality as in \cite{GouezelLalley:2013} and \cite{Gouezel:2014} in the setting of random walks. Therefore, we first consider the operator $\mathbb{H}_r$ defined by
\[\mathbb{H}_r(f_1, f_2) := \G_r(f_1\cdot \G_r(f_2)).\]
\begin{lemma} \label{lem:operator H_r} For $r \in (0,R)$,  $\mathbb{H}_r: \cH \times \cHL \to \cHL$.  If $T$ is transient, then $\mathbb{H}_R(\X_A f_1, \X_B f_2) \in \cHL$ for all $f_i \in \cH$ and $A,B \subset G$ finite.
\end{lemma}
\begin{proof} As $\G_r$ acts on $\cHL$ by  Proposition \ref{prop:action of G_r}, it remains to observe that $fg \in  \cHL$ for $f \in \cH $ and $g \in  \cHL$ in order to obtain that  $\mathbb{H}_r$ is well defined. For $r=R$, Proposition \ref{prop:finiteness of G_R} implies that $ \G_R(\X_B f_2) \in \cHL$. Hence, as $\X_A f_1 \G_R(\X_B f_2) \in \cH$ by finiteness of $A$, another application of  Proposition \ref{prop:finiteness of G_R} shows that $\mathbb{H}_R(\X_A f_1, \X_B f_2) \in \cHL$.
\end{proof}

\begin{lemma} Assume that $G$ is  a hyperbolic group. If $T$ is transient, then
\[\sup \left\{\sum_{|g|=k} \mathbb{H}_R(\X_g,\X_{\id})(x,\id) \Big| k \in \N, x\in \Sigma \right\} < \infty.\]
\end{lemma}

\begin{proof} The proof reads in verbatim as the one of Lemma 2.5 in \cite{Gouezel:2014}.
The only differences are that $\mathbb{H}_r$ now is an operator, well defined by Lemma \ref{lem:operator H_r}
and that one has to apply once Lemma \ref{lem:bounded distortion for the extended operator} in the estimates up to a constant.\end{proof}

We now adapt the principal estimate for obtaining the strong version of Ancona's inequalities in \cite{Gouezel:2014} to our operator setting. In order to do so, for $A \subset G$, set
\[ \G_r(f|A) = \sum_{n=0}^\infty r^n \cL^n\left(f \ \cdot \ \textstyle \prod_{k=1}^{n-1} {\X_A} \circ T^k\right) .
 \]
That is, $\G_r(f,A)$ corresponds to the sum of those paths, which stay in $\Sigma \times A$. In  analogy to \cite{Gouezel:2014}, the following estimate holds.

\begin{lemma} \label{lem:superexponential decay}
Assume that $G$ is hyperbolic, $T$ is a topologically transitive, transient extension of a Gibbs-Markov map of finite type and that $\G_R(\X_g)(x,\id) \asymp \G_R(\X_{\id})(x,g)$ independent of  $(x,g)\in \Sigma \times G$.
 Then there exists $n_0 \in \N$ and $\lambda >1$ such that for any $n > n_0$,  $g,\id,h \in G$ on a geodesic segment (in this order) with $d(g,\id)>n$, $d(\id,h) > n$, we have that
\[ \G_R(\X_h|B(\id,n)^c) (x,g) \leq 2^{-\lambda^n} \hbox{ for all }x \in \Sigma. \]
\end{lemma}

\begin{proof} As the proof reads in almost all parts in verbatim as the one of Lemma 2.6 in \cite{Gouezel:2014}, we again only indicate the necessary adaptions. The proof is based on a sequence of barriers $A_i$ such that the operator norm of $L_i: \ell^2(A_{i+1}) \to \ell^2(A_{i}) $ is smaller than $1/2$. These operators in the context considered in here have to be defined by, for $a \in A_i$ and $f \in \ell^2(A_{i+1})$,
\[ L_i(f)(a) = \sum_{b \in A_{i+1}} \left( \int \X_a \G_R(\X_b) d\mu \right) f(b).\]
Furthermore, it follows from $\G_R(\X_g)(x,\id)  \asymp \G_R(\X_{\id})(g,x)$ and Lemma \ref{lem:bounded distortion for the extended operator} that
\[\sum_{|g|=k}(\G_R(\X_g)(x,\id))^2 = C^{\pm 1} \sum_{|g|=k} \mathbb{H}_R(\X_g,\X_{\id})(x,\id).\]
As the method of proof in \cite{Gouezel:2014} allows to construct barriers the $A_i$ such that $\|L_i\|_{\ell^2}$ is arbitrary small, it is possible to absorb the constants $C$ and the one arising from a further application of Lemma \ref{lem:bounded distortion for the extended operator}.
\end{proof}

We are now in position to prove the main result of this section which generalizes the results for random walks with independent increments for cocompact Fuchsian groups in \cite[Th. 4.6]{GouezelLalley:2013} and hyperbolic groups in
\cite[Th. 2.9]{Gouezel:2014} to group extensions. The proof is an adaption of the arguments in \cite{GouezelLalley:2013}
to the setting of Green operators. In particular, as the proof of the exponential decay in the second part relies on a potential theoretic argument, it turns out to be necessary to replace the
concept of minimal harmonic functions by its dual, that is by minimal conformal measures as defined in Section \ref{sec:appendix}.

\begin{theorem}
 \label{theo:Ancona-Gouezel inequalities} Assume that $G$ is hyperbolic, $T$ is a topologically transitive, transient extension of a Gibbs-Markov map of finite type such that $\G_R(\X_g) (x,\id) \asymp \G_R(\X_{\id})(x,g)$, independent of  $(x,g)\in \Sigma \times G$.
\begin{enumerate}
\item \label{theo:Ancona-Gouezel inequalities - pt 1}
\textbf{Uniform Ancona inequality.} For any $D>0$, there exists $C>0$ such that for any $g,h\in G$, $x \in \Sigma$ and any $z \in G $ such that the distance between $z$ and a path from $g$ to $h$ of length $d(g,h)$ is smaller than $D$,
and any $r \in [1,R]$,
\begin{equation}\label{eq: strong ancona inequality} \G_r(\X_{h})(x,g) \leq C \G_r(\X_z \G_r(\X_{h}))(x,g).\end{equation}
\item \textbf{Gouëzel-Lalley inequality}. There exist $C>0$ and $\lambda \in (0,1)$ such that for any $r \in [1,R]$, for any $x,y \in \Sigma$ and for any $g,g',h,h' \in G$ in a configuration approximated by a tree as shown below, then 
\[ \left| \frac{\G_r(\X_h)(x,g)/\G_r(\X_h)(y,g')}{\G_r(\X_{h'})(x,g)/\G_r(\X_{h'})(y,g')}  -1 \right| \leq C\lambda^n  .\]
\begin{figure}[ht] 
   \centering
\def\svgwidth{0.6\textwidth}
\begingroup%
  \makeatletter%
  \providecommand\color[2][]{%
    \errmessage{(Inkscape) Color is used for the text in Inkscape, but the package 'color.sty' is not loaded}%
    \renewcommand\color[2][]{}%
  }%
  \providecommand\transparent[1]{%
    \errmessage{(Inkscape) Transparency is used (non-zero) for the text in Inkscape, but the package 'transparent.sty' is not loaded}%
    \renewcommand\transparent[1]{}%
  }%
  \providecommand\rotatebox[2]{#2}%
  \newcommand*\fsize{\dimexpr\f@size pt\relax}%
  \newcommand*\lineheight[1]{\fontsize{\fsize}{#1\fsize}\selectfont}%
  \ifx\svgwidth\undefined%
    \setlength{\unitlength}{584.75101776bp}%
    \ifx\svgscale\undefined%
      \relax%
    \else%
      \setlength{\unitlength}{\unitlength * \real{\svgscale}}%
    \fi%
  \else%
    \setlength{\unitlength}{\svgwidth}%
  \fi%
  \global\let\svgwidth\undefined%
  \global\let\svgscale\undefined%
  \makeatother%
  \begin{picture}(1,0.22414705)%
    \lineheight{1}%
    \setlength\tabcolsep{0pt}%
    \put(0,0){\includegraphics[width=\unitlength,page=1]{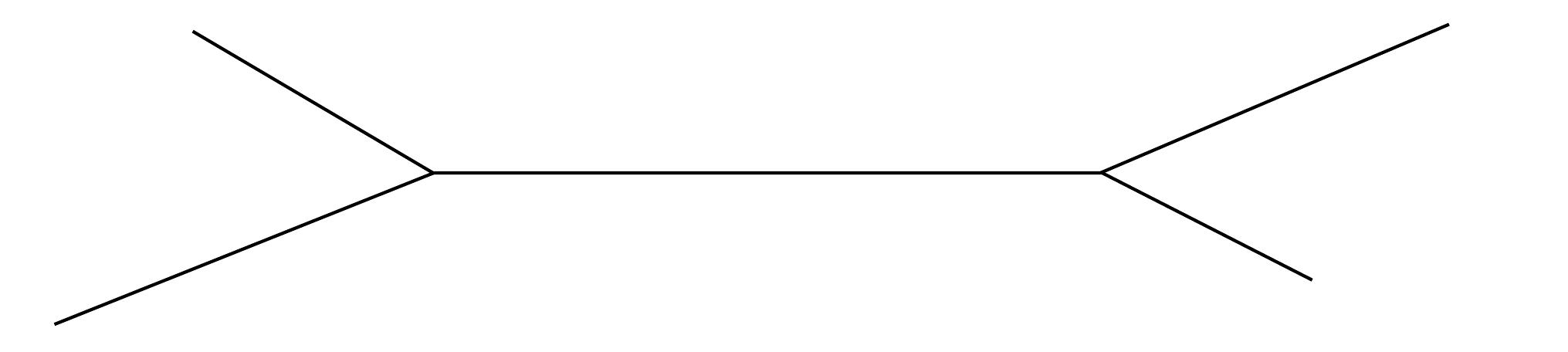}}%
    \put(0.84165561,0.03380153){\color[rgb]{0,0,0}\makebox(0,0)[lt]{\lineheight{0}\smash{\begin{tabular}[t]{l}$h$\end{tabular}}}}%
    \put(0.09099481,0.19411933){\color[rgb]{0,0,0}\makebox(0,0)[lt]{\lineheight{0}\smash{\begin{tabular}[t]{l}$g$\end{tabular}}}}%
    \put(-0.00216438,0.00885794){\color[rgb]{0,0,0}\makebox(0,0)[lt]{\lineheight{0}\smash{\begin{tabular}[t]{l}$g'$\end{tabular}}}}%
    \put(0.93312458,0.19468739){\color[rgb]{0,0,0}\makebox(0,0)[lt]{\lineheight{0}\smash{\begin{tabular}[t]{l}$h'$\end{tabular}}}}%
    \put(0,0){\includegraphics[width=\unitlength,page=2]{ancona-configuration-tex.pdf}}%
    \put(0.46023907,0.04080696){\color[rgb]{0,0,0}\makebox(0,0)[lt]{\lineheight{0}\smash{\begin{tabular}[t]{l}$>n$\end{tabular}}}}%
  \end{picture}%
\endgroup%
  \caption{Configuration of $g,g',h,h' \in G$}
  \label{ancona-configuration}
\end{figure}
\end{enumerate}
\end{theorem}

\begin{proof} We  closely follow the proofs of Theorems 4.1, 4.3 and 4.6 in \cite{GouezelLalley:2013}. In here, we also  make use of the following notational convention as in  \cite{GhysHarpe:1990}. 
Even $G$ is not necessarily a geodesic space, there is, for any pair $g,h \in G$ with distance $d = d(g,h)$, a path of length $d$ from $g$ to $h$ in the Cayley graph. By identifying the edges  with copies of $[0,1]$, one  obtains a continuous path $\gamma:[0,d] \to  G$  from $g$ to $h$ which is an isometry. As $\gamma$ not necessarily is uniquely determined, we refer to $\gamma$ as a geodesic from $g$ to $h$ and denote it by $[g,h]$. Furthermore, in order to slightly simplify the parameters, we assume that $G$ is $\delta/4$-hyperbolic in order to guarantee that triangles are $\delta$-thin, that is $[g,h]$ is always contained in a $\delta$-neighbourhood of $[g,w] \cup [h,w]$, for any configuration of $g,h,w \in G$.   

\medskip
\noindent\textsc{Part (i).}
For the proof of  part (i), assume that that $[g,h]$ is a geodesic segment in $G$, that $z \in [g,h] \setminus \{g,h\} $ and set $d:=d(g,h)$. Furthermore, let $\gamma:[0,d] \to  [g,h]$ refer to the isometry obtained by identifying the edges with copies of $[0,1]$ such that $\gamma(0) = g$ and $\gamma(d) = h$. We now construct finite sequences of $t_i ,s_i \in [0,d]$ and balls $B_i$ as follows. To begin, set $s_0 = 0$ and $t_0 = d$. The $s_i,t_i$  are then inductively constructed as follows (see Figure \ref{fig: strong ancona construction}).

\begin{enumerate}
\item If $t_i - s_i \leq 16$ then the induction stops. In fact, Figure \ref{fig: strong ancona construction} illustrates a possible last step in the iteration.
\item If $t_i - s_i > 16$ and $d(\gamma(s_i),z) \geq d(\gamma(t_i),z)$, then $s_{i+1}  =  s_i + (t_i - s_i)/4$ and $t_{i+1}  =  t_i$. As $s_{i+1} - s_i  = (t_i - s_i)/4 > 4$, there exists a ball $B_{i+1}$ is ball with center in $ \gamma((s_i,s_{i+1})) \cap G$ and radius in $\N$ such that $B_{i+1}$ covers $\gamma((s_i,s_{i+1}))$ up to two segments of total length at most 3.
\item If $t_i - s_i > 16$ and $d(\gamma(s_i),z) \leq d(\gamma(t_i),z)$, then $s_{i+1}  =  s_i$ and $t_{i+1}  =  t_i -  (t_i - s_i)/4$. As above, there exists a ball $B_{i+1}$ with center in $\gamma((t_{i+1},t_{i})) \cap G$ and radius $r_i \in \N$ such that $B_{i+1}$ covers $\gamma((t_{i+1},t_{i}))$ up to two segments of total length at most 3.
\end{enumerate}
Now assume that the induction stopped at step $n$. Then it is straightforward to see that $s_0 \leq s_i \cdots \leq s_{n} < t_n \leq t_{i+1} \cdots \leq t_0$, $t_i - s_i = d (3/4)^i$, that
$   (3/4)^i d  \leq  16 \, \diam B_{i+1} \leq 4 (3/4)^i d$ and that the distance between two adjacent balls is at most 4.
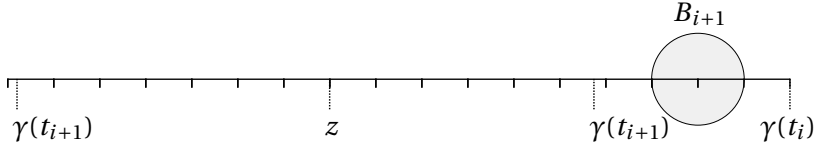
\begin{figure}[htbp] 
   \centering
   \def\svgwidth{0.8\textwidth}
\begingroup%
  \makeatletter%
  \providecommand\color[2][]{%
    \errmessage{(Inkscape) Color is used for the text in Inkscape, but the package 'color.sty' is not loaded}%
    \renewcommand\color[2][]{}%
  }%
  \providecommand\transparent[1]{%
    \errmessage{(Inkscape) Transparency is used (non-zero) for the text in Inkscape, but the package 'transparent.sty' is not loaded}%
    \renewcommand\transparent[1]{}%
  }%
  \providecommand\rotatebox[2]{#2}%
  \newcommand*\fsize{\dimexpr\f@size pt\relax}%
  \newcommand*\lineheight[1]{\fontsize{\fsize}{#1\fsize}\selectfont}%
  \ifx\svgwidth\undefined%
    \setlength{\unitlength}{1456.92328129bp}%
    \ifx\svgscale\undefined%
      \relax%
    \else%
      \setlength{\unitlength}{\unitlength * \real{\svgscale}}%
    \fi%
  \else%
    \setlength{\unitlength}{\svgwidth}%
  \fi%
  \global\let\svgwidth\undefined%
  \global\let\svgscale\undefined%
  \makeatother%
  \begin{picture}(1,0.14879454)%
    \lineheight{1}%
    \setlength\tabcolsep{0pt}%
    \put(0,0){\includegraphics[width=\unitlength,page=1]{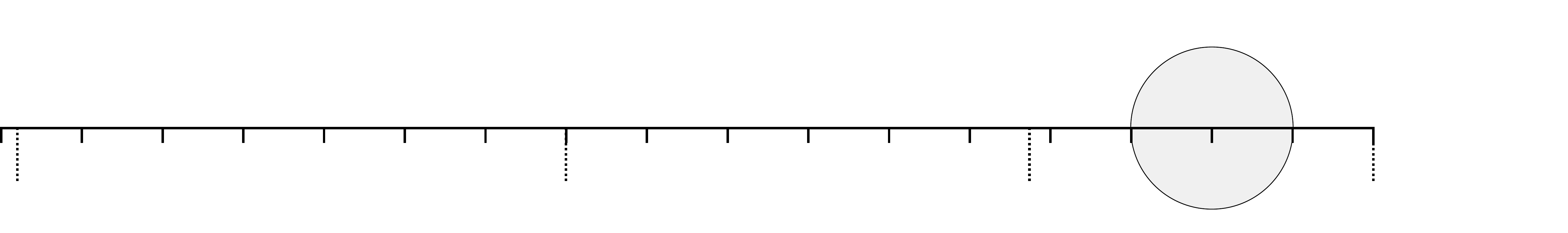}}%
    \put(0.74525461,0.13294887){\color[rgb]{0,0,0}\makebox(0,0)[lt]{\lineheight{1.25}\smash{\begin{tabular}[t]{l}$B_{i+1}$\end{tabular}}}}%
    \put(0.84254865,0.004253){\color[rgb]{0,0,0}\makebox(0,0)[lt]{\lineheight{1.25}\smash{\begin{tabular}[t]{l}$\gamma(t_i)$\end{tabular}}}}%
    \put(0.65156404,0.004253){\color[rgb]{0,0,0}\makebox(0,0)[lt]{\lineheight{1.25}\smash{\begin{tabular}[t]{l}$\gamma(t_{i+1})$\end{tabular}}}}%
    \put(0.0060256,0.004253){\color[rgb]{0,0,0}\makebox(0,0)[lt]{\lineheight{1.25}\smash{\begin{tabular}[t]{l}$\gamma(t_{i+1})$\end{tabular}}}}%
    \put(0.354534,0.004253){\color[rgb]{0,0,0}\makebox(0,0)[lt]{\lineheight{1.25}\smash{\begin{tabular}[t]{l}$z$\end{tabular}}}}%
  \end{picture}%
\endgroup%
   \caption{The construction of $B_{i+1}$}
   \label{fig: strong ancona construction}
\end{figure}

Now assume that $f \in \cH$ and that $B_i, B_j, B_k$ are
 three of these balls  in this order from the left to the right with respect to $\gamma$.
We now expand $\X_{B_i} \G_r(f \X_{B_k})$ according to the position of $B_j$ relative to $z$ as follows.
\begin{equation}\label{eq:induction for ancona} \X_{B_i} \G_r(f \X_{B_k}) = \X_{B_i} \G_r(f \X_{B_k}| B_j^c) +
\begin{cases}
  \X_{B_i} \G_r(\X_{B_j} \G_r(f \X_{B_k}) | B_j^c) & :\, B_j \hbox{ on the left of }z \\
  \X_{B_i} \G_r(\X_{B_j} \G_r(f \X_{B_k}| B_j^c)) & :\, B_j \hbox{ on the right of }z
\end{cases}
\end{equation}
That is, in the first case, we separate the orbits starting in $B_k$ and ending in $B_i$ at their last visit to $B_j$ whereas in the second case at their first visit to $B_j$.
We now apply this expansion inductively as follows. In the first case, we apply \eqref{eq:induction for ancona} to $\X_{B_j} \G_r(f \X_{B_k})$ with respect to $B_l$ between  $B_j$ and $B_k$,  and in the second case to  $\X_{B_i} \G_r( f^\ast \X_{B_j})$ with respect to   $B_l$ between  $B_i$ and  $B_j$, where $f^\ast = \G_r(f \X_{B_k}| B_j^c)$. In order to obtain a manageable expression, set
\[\G =\G_r, \quad \G_j = \G_r(\,\cdot \,| B_j^c), \quad {L}_j(f) := \G_r( \X_{B_j} \cdot  f | B_j^c), \quad
\quad {R}_j(f) := \X_{B_j} \cdot \G_r(f| B_j^c). \]
Furthermore, for $k \leq n$ assume that the $a_k{(i)} = 1, \ldots k$ ($i=1,\ldots k$) are determined
 by the order of the $B_i$ along the path $\gamma$ in the sense that $B_{a_k(i)}$ is followed by $B_{a_k(i+1)}$ etc. and that $\ell_k $ is given by $B_{a_k({\ell_k})} < z < B_{a_k({\ell_k+1})}$. Set
\begin{align*}
E_k & :=  \X_{\{g\}} \cdot L_{a_k(1)} \cdots  L_{a_k({\ell_k})} \circ \G \circ R_{a_k({\ell_k+1})} \cdots  R_{a_k(k)} (\X_{\{h\}}),\\
D_k & :=   \X_{\{g\}} \cdot L_{a_k(1)}   \cdots  L_{a_k({\ell_k})}  \circ \G_{k+1} \circ R_{a_k({\ell_k+1})} \cdots  R_{a_k(k)} (\X_{\{h\}}),\\
D_0 & := \X_{\{g\}} \cdot \G_1 (\X_{\{h\}}).
\end{align*}
In Figure \ref{fig: strong ancona construction 2}, typical orbits related to $D_3$ and $E_4$ are illustrated. That is, in the first case, the orbit is stopped at the first visit to $B_1$, then passes without hitting $B_4$ to the last visit to $B_3$ and through the last visit to $B_2$ to $g$, whereas in the second case, the orbit has to pass through $B_4$.
We now show by induction that $\X_{\{g\}}\G_r(\X_{\{h\}}) = E_k + \sum_{i=0}^{k-1} D_i$. If $k=1$, then $a_1(1)=1$ and $\ell_1 \in \{0,1\}$. In particular,
\[ E_1 = \begin{cases} \X_{\{g\}} \cdot \G_1( \X_{B_1} \cdot \G(\X_{\{h\}})) &:\, \ell_1=1 \\
 \X_{\{g\}} \cdot \G( \X_{B_1}  \cdot \G_1( \X_{\{h\}})) &:\, \ell_1=0 \end{cases}  \]
Hence, $\X_{\{g\}}\G_r(\X_{\{h\}}) = E_1 + D_0$ by \eqref{eq:induction for ancona}. In order to extend the result to any $k \leq n$, it suffices to apply \eqref{eq:induction for ancona} to
\[\X_{B_{a_k({\ell_k})}} \left(\G\left( \X_{B_{a_k({\ell_k}+1)}}\;\cdot\;\right) - \G\left(  \X_{B_{a_k({\ell_k+1})}}\;\cdot\; | {B_{a_k(\ell_{k+1})}}^c\right) \right) \]
in order to show that $E_k = E_{k+1} + D_k$, and, in particular, $\X_{\{g\}}\G_r(\X_{\{h\}}) = E_k + \sum_{i=0}^{k-1} D_i$ for all $k \leq n$ by induction.
\begin{figure}[htbp] 
   \centering
      \def\svgwidth{0.9\textwidth}
	\begingroup%
  \makeatletter%
  \providecommand\color[2][]{%
    \errmessage{(Inkscape) Color is used for the text in Inkscape, but the package 'color.sty' is not loaded}%
    \renewcommand\color[2][]{}%
  }%
  \providecommand\transparent[1]{%
    \errmessage{(Inkscape) Transparency is used (non-zero) for the text in Inkscape, but the package 'transparent.sty' is not loaded}%
    \renewcommand\transparent[1]{}%
  }%
  \providecommand\rotatebox[2]{#2}%
  \newcommand*\fsize{\dimexpr\f@size pt\relax}%
  \newcommand*\lineheight[1]{\fontsize{\fsize}{#1\fsize}\selectfont}%
  \ifx\svgwidth\undefined%
    \setlength{\unitlength}{1555.36565749bp}%
    \ifx\svgscale\undefined%
      \relax%
    \else%
      \setlength{\unitlength}{\unitlength * \real{\svgscale}}%
    \fi%
  \else%
    \setlength{\unitlength}{\svgwidth}%
  \fi%
  \global\let\svgwidth\undefined%
  \global\let\svgscale\undefined%
  \makeatother%
  \begin{picture}(1,0.35552119)%
    \lineheight{1}%
    \setlength\tabcolsep{0pt}%
    \put(0,0){\includegraphics[width=\unitlength,page=1]{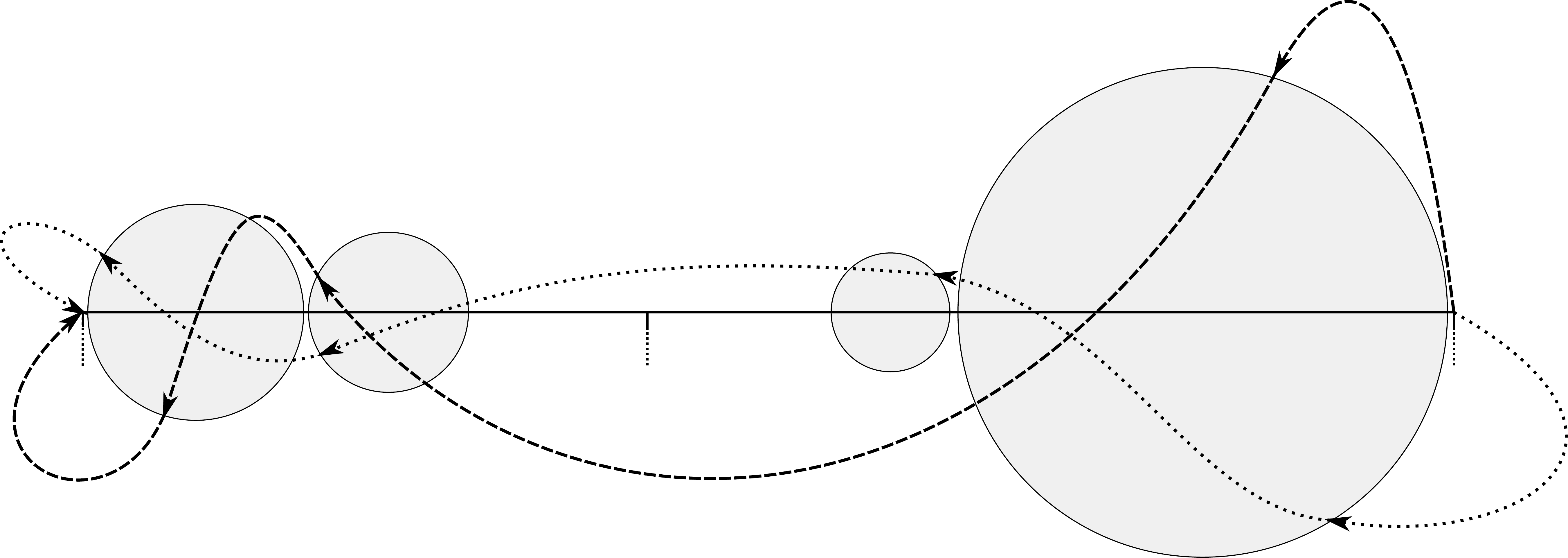}}%
    \put(0.10868318,0.23954458){\color[rgb]{0,0,0}\makebox(0,0)[lt]{\lineheight{1.25}\smash{\begin{tabular}[t]{l}$B_2$\end{tabular}}}}%
    \put(0.24128865,0.21945284){\color[rgb]{0,0,0}\makebox(0,0)[lt]{\lineheight{1.25}\smash{\begin{tabular}[t]{l}$B_3$\end{tabular}}}}%
    \put(0.55672895,0.20639321){\color[rgb]{0,0,0}\makebox(0,0)[lt]{\lineheight{1.25}\smash{\begin{tabular}[t]{l}$B_4$\end{tabular}}}}%
    \put(0.74760046,0.32493447){\color[rgb]{0,0,0}\makebox(0,0)[lt]{\lineheight{1.25}\smash{\begin{tabular}[t]{l}$B_1$\end{tabular}}}}%
    \put(0.40704549,0.09725489){\color[rgb]{0,0,0}\makebox(0,0)[lt]{\lineheight{1.25}\smash{\begin{tabular}[t]{l}$z$\end{tabular}}}}%
    \put(0.91737565,0.09725489){\color[rgb]{0,0,0}\makebox(0,0)[lt]{\lineheight{1.25}\smash{\begin{tabular}[t]{l}$h$\end{tabular}}}}%
    \put(0.04238044,0.09725489){\color[rgb]{0,0,0}\makebox(0,0)[lt]{\lineheight{1.25}\smash{\begin{tabular}[t]{l}$g$\end{tabular}}}}%
  \end{picture}%
\endgroup%
   \caption{Typical orbits for $D_3$ (slashed) and $E_4$ (dotted)}
   \label{fig: strong ancona construction 2}
\end{figure}
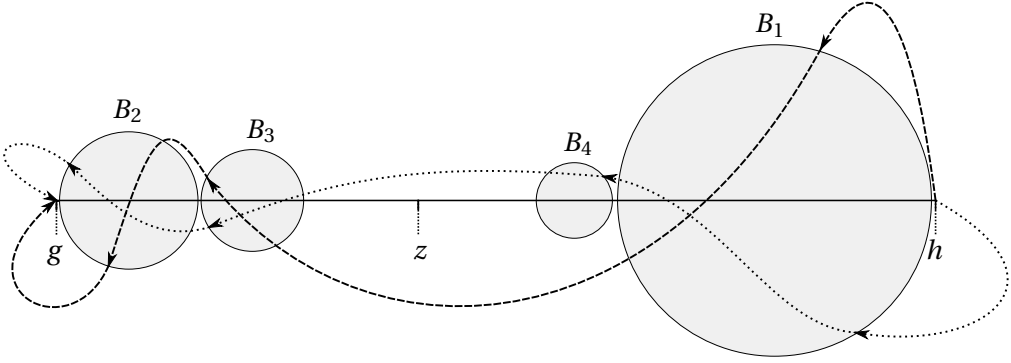
Now assume that $u \in B_{a_k({\ell_k})}$ and $v \in  B_{a_k({\ell_k}+1)}$. It follows from $\delta$-hyperbolicity that the distance from a geodesic segment $[u,v]$  to the center of $B_{k+1}$ is at most $\delta$. In particular, there is ball of radius $\diam B_{k+1}/2 - (\delta +1)$ with center in $[u,v] \cap G$ which is contained in $B_{k+1}$.
It now follows from Lemma \ref{lem:superexponential decay} and the construction that
\[\X_{\{u\}} \G_r( \1_{[a,{v}]}| B_{k+1}^c) \leq \X_{\{u\}} \G_r( \X_{\{v\}}| B_{k+1}^c) \leq 2^{-\lambda^{\diam(B_{k+1})/2  - (\delta +1)}} \leq 2^{- \lambda^{\frac{d}{32} \left(\frac{3}{4}\right)^k - (\delta +1)} } \]
for any $a \in \cW^1$,
provided that $d (3/4)^k\geq  32( n_0 + 1 +\delta)$. As $d\mu/d\mu\circ \theta$ is bounded away from zero, there
 exists $p \in (0,1)$ such that $\G_r( \1_{[a, v]}) (x,u) \gg p^{d(u,v)}$ for all $x\in \Sigma$.
Again by construction and the triangle inequality, it follows that $d(u,v) \leq |t_{k-1} - s_{k-1}| = d (3/4)^{k-1}$.
Hence, there exist $\alpha>1$, $\beta>1$ such that
\begin{equation} \label{eq:estimate-from-below-for-Green}
{\X_{\{u\}} \G_r( \1_{[a,{v}]} | B_{k+1}^c)} \leq \alpha^{-\beta^{d(3/4)^k} + d(3/4)^k} {\X_{\{u\}} \G_r(  \1_{[a,{v}]})}.\end{equation}
Now set $\alpha_{d,k} := \alpha^{-\beta^{d(3/4)^k} + d(3/4)^k}$ and suppose that  $f>0$ satisfies $\sup \{D_\alpha(\log f)(z,v): z \in \Sigma \} \leq \log C$. Then
\begin{align*} \X_{\{u\}} \G_r( \X_{\{v\}} \ f  | B_{k+1}^c)
& \leq
\sum_{a \in \cW^1}  \sup_{z \in [a,v]} f(x,v)\  {\X_{\{u\}} \G_r(\1_{[a,{v}]}  | B_{k+1}^c)} \\
& \leq \sum_{a \in \cW^1}  \sup_{x \in [a]} f(x,v)\  \alpha_{d,k}  \X_{\{u\}} \G_r(  \1_{[a,{v}]})
\leq C \alpha_{d,k}  \X_{\{u\}} \G_r( \X_{\{u\}}\  f   ) \end{align*}
%
By the Gibbs-Markov property of $\theta$, it therefore follows that
\begin{align*}
\sum_{i=1}^{k-1} D_i  & = \sum_{i=1}^{k-1} \sum_{u \in B_{a_i({\ell_i})},\atop v \in  B_{a_i({\ell_i+1})} }
\X_{\{g\}} \cdot L_{a_i(1)} \left(    \cdots  L_{a_i({\ell_i})}\left(
{\X_{\{u\}} \G_r\left( \X_{\{v\}}  R_{a_k({\ell_i+1})}\left( \cdots  (\X_{\{h\}})\right)| B_{i+1}^c\right)}\right)\right)\\
& \leq C_\varphi \sum_{i=1}^{k-1} \sum_{u \in B_{a_i({\ell_i})},\atop v \in  B_{a_i({\ell_i+1})} }
\alpha_{d,i} \X_{\{g\}} \cdot L_{a_i(1)} \left(    \cdots  L_{a_i({\ell_i})}\left(
\X_{\{u\}} \G_r\left( \X_{\{v\}}  R_{a_k({\ell_i+1})}\left( \cdots  (\X_{\{h\}})\right)\right)\right)\right)\\
& \leq \left( \textstyle \sum_{i=1}^{k-1}  \alpha_{d,i} \right) \X_{\{g\}}\G_r(\X_{\{h\}}).
\end{align*}
The next step relies on the fact that
$\alpha_{d,k}$ and $t_k - s_k$ are functions of $d(3/4)^k$, which allows to choose $M$ such that $t_k - s_k \geq M$ implies that $ \sum_{i=1}^{k-1}  \alpha_{d,i} \leq \frac{1}{2}$. For $k$ maximal with this property, it also follows that $B_{a_k(\ell_k)}$ and $B_{a_k(\ell_k+1)}$ are contained in a ball with center $z$ and radius $t_{k-1} - s_{k-1} + \diam B_{k-1}$. As the radius is uniformly bounded by a multiple of $M$, there exists $C> 0$ such that
\[   \X_{\{g\}} \cdots   \X_{\{u\}} \G_r ( \X_{\{v\}} \cdots \G_r(\X_{\{h\}}))  \leq
C   \X_{\{g\}} \cdots   \X_{\{u\}} \G_r ( \X_{\{z\}} \G_r( \X_{\{v\}} \cdots \G_r(\X_{\{h\}})))  \]
for all $u \in  B_{a_k(\ell_k)}$ and   $v \in B_{a_k(\ell_k+1)}$.
Putting these estimates together yields
\begin{align*} \X_{\{g\}}\G_r(\X_{\{h\}}) & = E_k + \sum_{i=0}^{k-1} D_i \leq
E_k + \sum_{i=0}^{k-1}  \alpha_{d,k}  E_i \leq E_k + \frac{1}{2} \X_{\{g\}}\G_r(\X_{\{h\}})  \\
& \leq \sum_{u \in  B_{a_k(\ell_k)},  v \in B_{a_k(\ell_k+1)}}
 \X_{\{g\}} \cdots   \X_{\{u\}} \G_r ( \X_{\{v\}} \cdots \G_r(\X_{\{h\}}))
  +   \frac{1}{2} \X_{\{g\}}\G_r(\X_{\{h\}}) \\
  & \leq  C \sum_{u,v}  \X_{\{g\}} \cdots   \X_{\{u\}} \G_r ( \X_{\{z\}} \G_r( \X_{\{v\}} \cdots \G_r(\X_{\{h\}})))   + \frac{1}{2} \X_{\{g\}}\G_r(\X_{\{h\}}) \\
  & \leq C  \X_{\{g\}}  \G_r ( \X_{\{z\}} \G_r(\X_{\{h\}})) + \frac{1}{2} \X_{\{g\}}\G_r(\X_{\{h\}}).
\end{align*}
Hence, $ \X_{\{g\}}\G_r(\X_{\{h\}}) \leq 2C  \X_{\{g\}}  \G_r ( \X_{\{z\}} \G_r(\X_{\{h\}}))$, proving \eqref{eq: strong ancona inequality} for $z \in [g,h]$.

Now assume that $z$ is $D$-close to the geodesic segment $[g,h]$. In particular, there exists $z' \in [g,h]$ with $d(z,z')< D$ and \eqref{eq: strong ancona inequality} holds with respect to $z'$.
Furthermore, note that $\{g: d(g,\id) \leq D \}$ is a finite set as $G$ is finitely generated. As $z^{-1}z' \in  \{g: d(g,\id) \leq D \}$, an application of  Lemma \ref{lem:bounded distortion for the extended operator} gives a uniform bound for
$|\log \G_r(\X_{z'} \G_r(\X_{h}))(x,g)/\G_r(\X_z \G_r(\X_{h}))(x,g)|$ which implies that \eqref{eq: strong ancona inequality} holds with respect to a different constant.

\medskip
\noindent{\textsc{An extension of Part (i).}} In order to deduce Part (ii) from Ancona's inequality, it is necessary to  extend part (i). In order to do so, observe that the induction relies on \eqref{eq:induction for ancona}, which is obtained through a decomposition of orbits. Hence, provided that $\Omega$ is a  set which contains $\bigcup_k B_k$, equation \eqref{eq:induction for ancona} generalizes to
\begin{equation*}\nonumber \label{eq:induction for ancona-2} \X_{B_i} \G_r(f \X_{B_k}|\Omega) = \X_{B_i} \G_r(f \X_{B_k}| B_j^c \cap \Omega) +
\begin{cases}
  \X_{B_i} \G_r(\X_{B_j} \G_r(f \X_{B_k}|\Omega) | B_j^c\cap \Omega) & :\, B_j \hbox{ left of }z \\
  \X_{B_i} \G_r(\X_{B_j} \G_r(f \X_{B_k}| B_j^c \cap \Omega)\Omega) & :\, B_j \hbox{ right of }z,
\end{cases}
\end{equation*}
which then implies that a version of the induction  $E_k + \sum_{i=0}^{k-1} D_i$ holds with respect to orbits which never leave $\Omega$. Moreover, this generalisation does not cause any problem with the application of
Lemma \ref{lem:superexponential decay} as the estimates in there only might get better. However, the estimate
\eqref{eq:estimate-from-below-for-Green} relies on the fact that there exists an orbit connecting $u$ and $v$. Therefore, it is also required that $\Omega$ contains a $M$-neighbourhood of the convex hull of  $\bigcup_k B_k$, where $M$ depends on the topological transitivity of $T$. That is, $M$ has to be chosen such that for any $a \in \cW$, $u \in B_k$, $v \in B_l$, there exists $x \in [a]$ and $n \in \N$ such that
the geodesic from $u$ to $v$ is contained in $\{u\kappa^j(x):0\leq j \leq n\}$, the orbit $\{u\kappa^j(x): 0\leq j \leq n\} \subset \Omega$ and $\log n \ll d(u,v)$. As the remaining assertions follow in verbatim, we obtain the following relative version of Part (i) by adding the trivial estimate,
 provided that $\Omega$ contains the  $M$-neighbourhood of the convex hull of  $\bigcup_k B_k$.
\begin{equation}\label{eq: strong ancona inequality - relative} \G_r\left(\X_z \G_r\left(\X_{h}|\Omega\right)|\Omega\right)(x,g) \leq  \G_r(\X_{h}|\Omega)(x,g) \leq C \G_r\left(\X_z \G_r\left(\X_{h}|\Omega\right)|\Omega\right)(x,g).\end{equation}

\medskip
\noindent{\textsc{Part (ii).}}
The adaption of the arguments in  in \cite{GouezelLalley:2013} for the proof of (ii) depends on the potential theory of conformal and excessive measures as developed in the appendix (Section \ref{sec:appendix}) of this article. In particular, it is necessary to anticipate the following notion from Section \ref{sec:geometry-Martin}, which also is the central object in Theorem \ref{theo:geometric-boundary} below.
We refer to a Radon measure $m$ as  {$\lambda$-excessive} or excessive if
$\cL^\ast(m) \leq \lambda m$, that is
$\cL^\ast(m)$ is absolutely continuous with respect to $m$ and $d\cL^\ast(m)/dm \leq \lambda$.
Moreover, we say that $m$ is conformal on $B$ if $\cL^\ast(m)|_B =  \lambda m|_B$.

We begin with an argument from geometry.
For $\xi,\eta \in G$ choose $k \in \N$ such that \[ D:= {d(\xi,\eta)}/{k} \geq 2 \max\{d(\id,\kappa(x)): x \in \Sigma \}.\]
 For $1\leq j \leq k$, let $z_j \in G$ refer the closest point on the geodesic arc from $\xi$ to $\eta$ with
 $(z_j \cdot \xi)_\eta > jD + D/2$ and set
\[ \Omega_j :=  \left\{ h \in G :  (h \cdot \xi)_\eta \geq jD \right\}, \; \Lambda_j :=  \left\{ h \in \Omega_j :
 (h \cdot \xi)_\eta  \leq D/2 +jD \right\}. \]
Observe that $ \Omega_j \supset
\Omega_{j+1}$ and that for $h \in \Omega_j$, the geodesic  from $h$ to $\eta$ passes through $B(z_j,\delta)$ by the thin triangle property. For  $h \in \Omega_j$ and
$g \in \Omega_{j}^c$, it follows from the construction that $(h \cdot \xi)_\eta > (g \cdot \xi)_\eta$. By  approximation by a tree, the geodesic from $h$ to $g$ has to pass through $B(z_j,4\delta)$. If, in addition, $h \in \Omega_{j+1}$, it follows from the choice of $D$ that any orbit from $h$ to $g$ has to pass through $\Lambda_j$, say at $z$. By a further approximation by a tree, also the geodesic from $h$ to $z$ visits $B(z_j,4\delta)$.

This geometrical construction now allows to deduce the following estimates. By decomposing orbits with respect to the last visit to $\Lambda_j$, it follows from \eqref{eq: strong ancona inequality - relative} and an extension of Lemma \ref{lem:bounded distortion for the extended operator} to $\G_r(\cdot|\Omega_j)$ that there exists $c\geq 1$ such that, for any $\omega \in \Sigma$,
\begin{align*}
\X_g \G_r(\X_h |\Omega_j) &= \X_g \sum_{z \in\Lambda_j} \G_r( \X_z \G_r(\X_h |\Omega_j)\; | \Omega_j \setminus \Lambda_j) \\
& = c^{\pm 1} \X_g \sum_{z \in\Lambda_j}
 \G_r(
 \X_z  \G_r(\X_{z_j} \G_r(\X_{h}|\Omega_j)  |\Omega_j)
 | \Omega_j \setminus \Lambda_j)\\
 & = c^{\pm 2} \G_r(\X_{h}|\Omega_j) (\omega,z_j) \,\cdot \, \X_g \sum_{z \in\Lambda_j}
 \G_r(
 \X_z  \G_r(\X_{z_j}  |\Omega_j)
 | \Omega_j \setminus \Lambda_j)
 \\
 & = \G_r(\X_{h}|\Omega_j) (\omega,z_j)  \,\cdot \, \X_g \G_r(\X_{z_j} |\Omega_j)
 ,
\end{align*}
Given $A \subset G$, set $\mathbf{X}_A := \{(x,g): x \in \Sigma, g \in A\}$. 
Now assume that, for some  $1\leq j \leq k$, $m$ is a Radon measure which is $1/r$-conformal and non-trivial on $\mathcal{X}_{\Omega_j}$ such that $m \left( \bigcap_{n=0}^\infty T^{-n}(\mathcal{X}_{\Omega_j})\right)=0$.
In particular, we have that
\[A_k:= T^{-k}(\mathcal{X}_{\Omega_j}^c) \cap \bigcap_{n=0}^{k-1} T^{-n}(\mathcal{X}_{\Omega_j}), \quad k=1,2,\ldots\]
is a partition of $\mathcal{X}_{\Omega_j}$ up to a set of measure zero. Hence, for $h \in {\Omega_{j+1}}$,
\begin{align*}
m(\mathcal{X}_h) & = \sum_{k=1}^\infty \int_{A_k} \X_h dm
	  = \sum_{k=1}^\infty \int \X_{\Omega_j^c} r^{k-1}\cL\left(\X_{\Omega_j} \cL ( \cdots \cL(\X_h)\cdots )\right)dm\\
	& = \frac1r \int \X_{\Omega_j^c} \G_r(\X_h|{\Omega_j}) dm
	= \frac{c^{\pm 2}}{r}   \G_r(\X_{h}|\Omega_j) (\omega,z_j)  \int \X_{\Omega_j^c}  \G_r(\X_{z_j} |\Omega_j)  dm \\
	& = c^{\pm 2}\G_r(\X_{h}|\Omega_j) (\omega,z_j)  \,\cdot \, m(\mathcal{X}_{z_j}).
\end{align*}
Setting $h=\xi$, it follows that $m(\mathcal{X}_{z_j}) = c^{\pm 2}m(\mathcal{X}_\xi)/ \G_r(\X_{\xi}|\Omega_j) (\omega,z_j)$.  Hence, for $\nu_j$ defined through
\[ \int f d\nu_{j} := c^{-4} \frac{\G_r( f |\Omega_j) (\omega,z_j)}{\G_r(\X_{\xi}|\Omega_j) (\omega,z_j)},\]
we have that, for any $h \in {\Omega_{j+1}}$,
\begin{equation}\label{eq:contraction}
c^{-4}  m(\mathcal{X}_h) \leq  {m(\mathcal{X}_\xi)} \nu_{j}(\mathcal{X}_h) \leq c^4 m(\mathcal{X}_h).
\end{equation}
Observe that in most cases, $m$ and $\nu_j$ are non-singular with respect
to each other. In order to apply  \eqref{eq:contraction}  also to $m=\nu_{j-1}$, note
that  $\nu_j$ is  $1/r$-conformal on $\left(T^{-1}(\mathcal{X}_{\Omega_j})\cap \mathcal{X}_{\Omega_j} \right) \setminus \{(\omega,z_j)\} \supset \mathcal{X}_{\Omega_{j+1}}$ as
\[ \G_r(\cL(f)|\Omega_j) = \frac1r \left(\G_r(f|\Omega_j) -f  \right) + \G_r(\X_{\Omega_j^c}\cL(f)|\Omega_j) - \X_{\Omega_j^c}\cL(f),  \]
and that, by construction, $\nu_j \left( \bigcap_{n=0}^\infty T^{-n}(  \mathcal{X}_{\Omega_{j+1}})\right)=0$.
It is worth noting that these two properties are needed for the lower bound of $(\mu_i - \nu)(\mathcal{X}_h)$ below whereas the upper bound is independent from this.

After these preparations, we are now in position to prove part (ii). In order to do so, for $\alpha := 1-c^{-4}$, $x_1,x_2 \in \Sigma$ and $g_1,g_2 \in \Omega_1^c$, let $\mu_1,\mu_2,\nu$ refer to the Radon measures defined by
\[ \mu_i(f) := \frac{\G_r(f)(x_i,g_i)}{\G_r(\X_\xi)(x_i,g_i)}, \quad
  \nu := \sum_{j=1}^{k-1} \alpha^{j-1} \nu_j.\]
By inductively applying \eqref{eq:contraction} to $m = \mu$ for the estimate from above and $m=\nu_j$ for the estimate from below, it follows that, for $h \in \Omega_k$,
\begin{align*}
(\mu_i - \nu)(\mathcal{X}_h)  & =  (\mu_i  - \nu_1)(\mathcal{X}_h) -
  \sum_{j=2}^{k-1} \alpha^{j-1} \nu_j(\mathcal{X}_h)  \leq  \alpha \left( \mu_i(\mathcal{X}_h) -  \sum_{j=2}^{k-1} \alpha^{j-2} \nu_j(\mathcal{X}_h) \right) \\
  & \leq \alpha^{k-1} \mu_i(\mathcal{X}_h),\\
  (\mu_i - \nu)(\mathcal{X}_h)  & =  (\mu_i  - \nu_1)(\mathcal{X}_h)  -
  \sum_{j=2}^{k-1} \alpha^{j-1} \nu_j(\mathcal{X}_h)
  \geq \alpha  \left( \nu_1(\mathcal{X}_h) -  \sum_{j=2}^{k-1} \alpha^{j-2} \nu_j(\mathcal{X}_h) \right) \\
 & \geq   \alpha^{k-1} \nu_{k-1}(\mathcal{X}_h) \geq 0.
\end{align*}
Moreover, note that \eqref{eq:contraction} implies that $\mu_1(\mathcal{X}_h) \asymp \mu_2(\mathcal{X}_h)$. Hence,
\begin{align*}
\left| \frac{\mu_1(\mathcal{X}_h)}{\mu_2(\mathcal{X}_h)}  -1 \right| & = \left| \frac{\mu_1(\mathcal{X}_h)- \mu_2(\mathcal{X}_h)}{\mu_2(\mathcal{X}_h)}   \right|  =  \left| \frac{(\mu_1- \nu )(\mathcal{X}_h)+ ( \mu_2 -\nu) (\mathcal{X}_h) }{\mu_2(\mathcal{X}_h)}   \right|\\
& \leq \alpha^{k-1} \frac{\mu_1(\mathcal{X}_h) + \mu_2(\mathcal{X}_h)}{\mu_2(\mathcal{X}_h)} \ll \alpha^{k-1} \frac{\mu_2(\mathcal{X}_h) + \mu_2(\mathcal{X}_h)}{\mu_2(\mathcal{X}_h)} \ll \alpha^{k},
\end{align*}
which is part (ii) of the theorem for $h'= \xi$ and $\lambda := \alpha^{1/D}$. The remaining assertion, that is $h'\in \Omega_k$ easily follows from this.
\end{proof}

\section{Geometry of the Martin boundary}\label{sec:geometry-Martin}

Theorem \ref{theo:Ancona-Gouezel inequalities} has immediate implications for a boundary theory of group extensions as it indicates what might be the canonical notion of a Martin boundary through a geometrization. Namely,
the second estimate in Theorem \ref{theo:Ancona-Gouezel inequalities} allows to obtain a
 geometrization by a bi-Hölder equivalence of the Martin boundary with the visual boundary of $G$.

The boundary of a hyperbolic group $G$ is defined as follows (see, e.g. \cite{GhysHarpe:1990b}). A sequence $(g_n)$ is said to \emph{converge at infinity} if  $\lim_{m,n \to \infty} (g_n \cdot g_m)_p = \infty$ for  some $p\in G$, and we say that  $(g_n)$ and $(h_n)$ converge to the same limit at infinity if  $\lim_{n \to \infty} (g_n,h_n)_p = \infty$  for  some $p\in G$. The boundary $\partial G$ of $G$ is then defined as the set of equivalence classes of this relation, and, in particular, for a convergent sequence $(g_n)$, the limit is defined as its associated
equivalence class. Moreover, these definitions do not depend on the choice of $p$.

In order to define the visual metric on $\partial G$, for $\xi,\eta \in \partial G$, let
\[ (\xi \cdot \eta) := \sup\left\{ \liminf_{m,n \to \infty} (g_m \cdot h_n)_\mathbf{o} \;:\; g_n \to \xi, h_m\to \eta \right\}.\]
As shown in \cite{GhysHarpe:1990b}, if $G$ is $\delta$-hyperbolic, then $(\xi \cdot \eta) - 2\delta \leq \liminf_{m,n} (g_m \cdot h_n)_\mathbf{o} \leq
(\xi \cdot\eta)$, for any approximating sequences $(g_n)$ and $(h_m)$. Furthermore, for $\lambda_{\hbox{\tiny visual}} \in (\sqrt[2\delta]{1/2},1)$, it is shown in  \cite{GhysHarpe:1990b} that
\begin{equation} \label{eq:parameter_for_Gromov_boundary}  r(\xi,\eta) := \lambda_{\hbox{\tiny visual}}^{(\xi \cdot \eta) }, \;  d_{\hbox{\tiny visual}}(\xi,\eta) := \inf \left\{  \sum_{k=1}^{n-1} r(x_k,x_{k+1}) : n \in \N, x_k \in \partial G, x_1 = \xi, x_{n}=\eta \right\}, \Xi: \partial G \to \mathcal{M}_r
\end{equation}
defines a metric on $\partial G$ and that there exists $C\in (0,1)$ such that $ C r(\xi,\eta)   \leq d(\xi,\eta)\leq r(\xi,\eta)$
for all $\xi,\eta \in \partial G$. Moreover,  $\partial G$ is compact with respect to this metric. For this choice of $\lambda$, we refer to $d_{\hbox{\tiny visual}}$ as the \emph{visual metric} on $\partial G$.

The following definition is inspired by the classical construction of the Martin boundary as it gives rise to a continuous extension of the Green operators. For $h \in G$ and $r < R$, let
\begin{align*} \mathbb{K}_r(h,\cdot): \Sigma \times G \to \R,\;  (x,g) \mapsto  \frac{\mathbb{G}_r (\X_h)(x,g)}{ \mathbb{G}_r (\X_{\id})(x,g)}
\end{align*}
and note that $ \mathbb{K}_r(h,\cdot)$ is a bounded function by Lemma \ref{lem:bounded distortion for the extended operator} for each $h\in G$. Now assume that $(g_n)$ is a sequence in $G$. We refer to   $(g_n)$  as \emph{unbounded}, written as $|g_n| \to \infty$, if $(g_n)$ leaves any finite subset of $G$ infinitely often.
Furthermore, let
\begin{align*} M_r &:= \left\{ (x,g) \in \Sigma \times G : |\kappa^n(x)| \to \infty \hbox{ and } \lim_{n \to \infty}  \mathbb{K}_r(h,T^n(x,g))   \hbox{ exists for all } h \in G\right\}
\end{align*}
and $(x,g) \sim (\tilde{x},\tilde{g})$ if and only if
$\lim_{n \to \infty}  \mathbb{K}_r(h,T^n(x,g)) =  \lim_{n \to \infty} \mathbb{K}_r(h,T^n(\tilde{x},\tilde{g}))$ for all $h \in G$. In analogy to the theory known from random walks, we refer to
$\mathcal{M}_r := M_r/_\sim$ as the \emph{Martin boundary} of the group extension $(X,T)$.
As a consequence of Theorem \ref{theo:Ancona-Gouezel inequalities}, one obtains the following relation of $\partial G$ and $\mathcal{M}_r$.
\begin{proposition} \label{prop:Xi_is_well_defined}
  Assume that $G$ is hyperbolic, $T$ is a topologically transitive, transient  extension of a Gibbs-Markov map of finite type and that $\G_R(\X_g)(x,\id) \asymp \G_R(\X_{\id})(x,g)$, independent of  $(x,g)\in \Sigma \times G$. For $r \leq R$, the following holds.
\begin{enumerate}
\item For $(x,g)\in \Sigma \times G$ such that $(g\kappa^n(x))$ converges at infinity, we have that $(x,g) \in M_r$. Moreover, for $(\tilde{x},\tilde{g})\in \Sigma \times G$ such that $(g\kappa^n(x))$ and  $(\tilde{g}\kappa^n(\tilde{x}))$ converge to the same limit at infinity in the sense of Gromov, it follows that $(x,g) \sim (\tilde{x},\tilde{g})$.
\item For each sequence $(g_n)$ which converges at infinity, there exists  $x\in \Sigma $ such that $(\kappa^n(x))$ and  $(g_n)$ converge to the same limit at  infinity.
\end{enumerate}
\end{proposition}

\begin{proof}
For the proof of (i), observe that, for $h \in G$ and $N$ sufficiently large,
the second part of Theorem \ref{theo:Ancona-Gouezel inequalities} is applicable to $\id$, $h$, $g\kappa^k(x)$ and $g\kappa^l(x)$, for $k,l \geq N$. As $n$ in the statement of the theorem can be written as
\begin{equation} \label{eq:estimate-as-gromov-product} n = \frac{1}2 \left((g\kappa^k(x) \cdot g\kappa^l(x))_{\id} + (g\kappa^k(x) \cdot g\kappa^l(x))_h - d(h,\id)\right),\end{equation}
with $(g \cdot \tilde{g})_h$ referring to the Gromov product with base $h$ and $d$ the word metric on $G$, it immediately follows that
$((g\kappa^k(x) \cdot g\kappa^l(x)) \to
\infty$ implies that $\log  \mathbb{K}_r(h,T^k(x,g))$ is a Cauchy sequence for all $r \leq R$. The second part follows by substituting $g\kappa^l(x)$ with $\tilde{g}\kappa^n(\tilde{x})$ in \eqref{eq:estimate-as-gromov-product}.

Assertion (ii) follows from the fact that the transitivity of $T$ allows to construct $x \in \Sigma$ such that $(\kappa^n(x))$ stays uniformly close to the piecewise geodesic arc with vertices $(g_n)$. It is then well known that $(\kappa^n(x))$ and  $(g_n)$ have the same limit.
\end{proof}

As an immediate corollary of the result, the  application
\[\Xi: \partial G \to \mathcal{M}_r, \; \eta \to \left\{(x,g) \in M_r: \lim_{n \to \infty} g\kappa^n(x) = \eta \right\}\big/_\sim \]
is well defined. In order to show that the map is invertible,
we apply ideias by Ancona and Shwartz in
\cite{Ancona:1987,Shwartz:2019a} to our setting  as follows. First observe that,
for $\sigma  \in  \mathcal{M}_r$ and $(x,g)$ in the equivalence class $\sigma$,
 \[ \mathbb{K}_r: G \times  \mathcal{M}_r \to \R,\quad  (h,\sigma) \mapsto  \lim_{n \to \infty}  \mathbb{K}_r(h,T^n(x,g))  \]
 is well defined and extends the definition of $\mathbb{K}_r$, but, in contrast to the setting of Markov chains, the function $h  \mapsto \mathbb{K}_r(h,\sigma)$ is not related to an $r$-harmonic function. However, by assuming transience,
the definition of $\mathbb{K}_r$ extends to $\cHL\times X$ for $r \leq R$ (see Prop. \ref{prop:finiteness of G_R}). In particular, as $h$ is identified with $\X_h$, a calculation shows that
\begin{align}\nonumber \label{eq:conformal} \mathbb{K}_r(\cL(\X_h),\sigma) & = \lim_n \mathbb{K}_r(\cL(\X_h),T^n(x,g)) = \lim_n \frac{1}{r}\left(\mathbb{K}_r(\X_h,T^n(x,g)) - \frac{\X_h(T^n(x,g))}{\G_r(\X_{\id},T^n(x,g))} \right) \\
 & = \frac{1}{r}\mathbb{K}_r(\X_h,\sigma),
\end{align}
where the last equality follows from the fact that $(g\kappa_n(x))$ leaves every finite subset of $G$. This identity implies that the canonical approach in here is to consider conformal and excessive measures as developed in the section on potential theory below. In here, we refer to a Radon measure $m$ as \emph{$\lambda$-excessive} if $\cL^\ast(m) \leq \lambda m$, and as  \emph{$\lambda$-conformal} if $\cL^\ast(m) =  \lambda m$. Moreover, a conformal measure $\mu$ is referred to as \emph{minimal} if any conformal measure $m$ with $m \leq \nu$ is a multiple of $\mu$.
The following theorem identifies $\partial G$ with minimal, conformal measures.

\begin{theorem}\label{theo:geometric-boundary} Assume that $G$ is hyperbolic, $T$ is a topologically transitive, transient  extension of a Gibbs-Markov map of finite type and that $\G_R(\X_g)(x,\id) \asymp \G_R(\X_{\id})(x,g)$, independent of  $(x,g)\in \Sigma \times G$ and that  $r \leq R$. Then the following holds.
\begin{enumerate}
\item If  $\sigma  \in  \mathcal{M}_r$ and $(x,g)$ is an element of the equivalence class $\sigma$ and $f \in \cHL$,  $f \geq 0$, then
\[ \mu_\sigma (f) :=  \lim_{n \to \infty} \mathbb{K}_r(f,T^n(x,g)),\]
always exists and defines a $1/r$-conformal, minimal  measure.
Moreover, any $1/r$-conformal, minimal  measure is obtained in this way.
\item For any conformal measure $\mu$, there exists a uniquely defined finite measure $\nu$ on $\partial G$ such that
$d\mu = d\mu_\sigma d\nu(\sigma)$, that is, for any $f \in \cHL$,
\[ \mu(f) =  \int_{\partial G} \mathbb{K}_r(f,\sigma) d\nu(\sigma).\]
\item If $\tilde{\sigma} \neq {\sigma}$, then $\lim_{\gamma \to \sigma}  \mu_\sigma(\X_{\gamma}) = \infty $ and $\lim_{\gamma \to \sigma}  \mu_{\tilde{\sigma}}(\X_{\gamma}) = 0$. In particular, $\mu_{\tilde{\sigma}} \neq \mu_{{\sigma}}$.
\item If $\tilde{\sigma} \neq {\sigma}$, then  $g \to \log  \mu_\sigma(\X_{g}) / \mu_{\tilde\sigma}(\X_{g})$ extends to a continuous function on $\overline{G} \setminus \{\tilde{\sigma},{\sigma}\}$. Furthermore, if  $g,h \in \overline{G}$ and $\sigma,\tilde{\sigma} \in \partial G$ are in configuration as in figure \ref{ancona-configuration}, then, with $C, \lambda$ as in  Theorem \ref{theo:Ancona-Gouezel inequalities},
\[ \left| \frac{\mu_\sigma(\X_{g})}{\mu_{\tilde{\sigma}}(\X_{g})} \cdot  \frac{\mu_{\tilde{\sigma}}(\X_{h})}{\mu_\sigma(\X_{h})} -1 \right| \leq C \lambda^n. \]
\item The Green operator $\G_r\left(\X_{g}\right)$ converges to $0$ uniformly and exponentially fast, that is
\[ \limsup_{n \to \infty} \max_{y \in \Sigma, |\gamma|=n} \sqrt[n]{\G_r\left(\X_{g}\right)\left(y, \gamma \right)} < 1.\]
\end{enumerate}
\end{theorem}

\begin{proof} The strategy is as follows. We begin with the construction of an accumulation point of $\mu_n(f) := \mathbb{K}_r(f,T^n(x,g))$ with respect to a particular $x$ (Step 1) and then apply the Ancona-Gou{\"e}zel inequality in order to identify a region where the
the limit  is comparable to a reduced measure for a given conformal measure (Step 2).
We then conclude from this description that the accumulation point is minimal (Step 3) and therefore unique, which implies convergence for each $x$ in the equivalence class of $\sigma$ (Step 4). An application of the argument in Step 3 then allows to prove assertions (iii-v) of the theorem (Steps 5 and 6). In Step 7, it is then shown how to deduce the remaining assertion from the work of Shwartz in \cite{Shwartz:2019a}.

\medskip
\noindent\textsc{Step 1. Accumulation points}. Assume that $(x,\id)$ is as in Proposition \ref{prop:Xi_is_well_defined}, that is there exists a subsequence $(n_k)$ and $ {x}\in \Sigma$ such that $(\kappa^{n_k}({x})) \to \sigma$ and $\kappa^{n_k}({x})$ stays within a bounded distance to a geodesic half ray $[\id,\sigma]$.

In order to construct a limit measure, observe that $\mu_{n_k}(f)$
defines a measure for each $k\in \N$, and moreover, as $\lim_{n \to \infty} \mu_n(\X_h) = \mathbb{K}_r(\X_h,\sigma)$, for all $h \in G$. Hence, by compactness of $\Sigma \times \{h\}$ and a diagonal argument, there exists a further subsequence, also denoted by $(n_k)$, such that $\mu  := \lim_k \mu_{n_k}$ converges weakly on compact sets. Moreover, \eqref{eq:conformal} implies that  $\mu$ is conformal. In particular, it follows from  bounded distortion that, for each $w \in \cW^n$,
\begin{equation}\label{eq:comparability-of-boundary-measure} \mu([w,h]) \asymp \varphi_w r^{n} \mu (\X_{h\kappa^n(w)}) = \varphi_w r^{n} \lim_n \mathbb{K}_r(\X_h,T^n(x,g)) .\end{equation}

\medskip
\noindent\textsc{Step 2. Reduced measures and the Ancona inequality.}
Fix $g \in G$. By construction, $\kappa^{n_k}({x}) \to \sigma$ and $\kappa^{n_k}({x})$ stays within a bounded distance to the geodesic half ray $[\id,\sigma]$. Hence, there exists $K$ such that $\kappa^{n_k}({x})$ stays within a bounded distance to the geodesic half ray $[g,\sigma]$ for any $k \geq K$.
Hence, \eqref{eq: strong ancona inequality} of Theorem \ref{theo:Ancona-Gouezel inequalities} is applicable to $g,\kappa^{n_k}(x), \kappa^{n_l}(x)$ for $K<k<l$. This implies after dividing  by $\G_r(\X_{\id})(T^{n_l}(x,\id)$ and applying Lemma \ref{lem:bounded distortion for the extended operator} for any $y \in \Sigma$ that
\begin{align}
\nonumber \mu(\X_g) & = \lim_{l \to \infty }  \mathbb{K}_r\left(\X_g,T^{n_l}(x,g)\right) =  \lim_{l \to \infty } \frac{\G_r\left(\X_{g}\right)(T^{n_l}(x,\id) )}{\G_r(\X_{\id})(T^{n_l}(x,\id) )} \\
\nonumber & \asymp \lim_{l \to \infty } \frac{\G_r\left(\X_{\kappa^{n_k}(x)} \G_r(\X_{g})\right)(T^{n_l}(x,\id))}{\G_r\left(\X_{\id}\right)(T^{n_l}(x,\id) )}
 = \int  \X_{\kappa^{n_k}(x)} \G_r(\X_{g})  d\mu \\
\label{eq:ancona-for-conformal}& \asymp \G_r\left(\X_{g}\right)\left(y, \kappa^{n_k}(x)\right) \; \mu\left(\X_{\kappa^{n_k}(x)}\right).
\end{align}
Set $h= \kappa^{n_k}(x)$.  As $y$ is arbitrary, integrating with $\nu(\X_h)^{-1} \nu|_{\X_h}$ for some measure $\nu$ gives
\begin{align}\label{eq:assintotica de mu} \mu(\X_g) \asymp  \frac{\mu(\X_h)}{\nu(\X_h)} \int \G_r(\X_{g}) d\nu|_{\X_h}  = \frac{\mu(\X_h)}{\nu(\X_h)}  \; (\G_r^\ast(\nu|_{\X_h}))(\X_g).  \end{align}
Now assume that $a \in \cW^1$. By Theorem \ref{theo:domination}, $\G_r^\ast(\nu|_{[a,h]})$ is already reduced.
If, in addition, $\nu$ is a conformal measure, then Theorem \ref{theo:reduced_measure} implies that the reduced measure is obtained by applying the operator  $ \mathcal{F}_{[a,h]}^\ast$. Hence,
\begin{align*} \G_r^\ast(\nu|_{[a,h]})(\X_g)  & =  R_{[a,h]}(\G_r^\ast(\nu|_{[a,h]}))(\X_g) = \mathcal{F}_{[a,h]}^\ast \circ \G_r^\ast(\nu|_{[a,h]})(\X_g) \\
 & = \int \1_{[a,h]}  \G_r(\1_{[a,h]} \mathcal{F}_{[a,h]}(\X_g)) d\nu \\
& \leq \sup_{z \in [a,h]}  \G_r(\1_{[a,h]})(z)  \sup_{z \in [a,h]}  \mathcal{F}_{[a,h]}(\X_g)(z)  \nu([a,h]) \\
& \stackrel{(\dagger)}{\leq} C_\varphi \sup_{z \in [a,h]}  \G_r(\1_{[a,h]})(z) \int \mathcal{F}_{[a,h]}(\X_g) d\nu \\
& \stackrel{(\ddagger)}{=} C_\varphi \sup_{z \in [a,\id]}  \G_r(\1_{[a,\id]})(z)  \; \cdot \; R_{[a,h]}(\nu)(\X_g) \ll  R_{[a,h]}(\nu)(\X_g),
\end{align*}
where $(\dagger)$ follows from bounded distortion of $\varphi$ and $(\ddagger)$ from the
 fact that $\varphi (x,g)$ does not depend on the second coordinate.
In particular, by construction of $\mathcal{F}_{\X_h}$,
\begin{align*}  \G_r^\ast(\nu|_{\X_h})(\X_g) & = \sum_{a \in \cW^1}  \G_r^\ast(\nu|_{[a,h]})(\X_g)
\ll \sum_{a \in \cW^1} \mathcal{F}_{[a,h]}^\ast(\nu)(\X_g) \leq \mathcal{F}_{\X_h}^\ast(\nu)(\X_g) = R_{\X_h}(\nu)(\X_g)
 .\end{align*}
Hence, $R_{\X_h}(\nu)(\X_g) \asymp \G_r^\ast(\nu|_{\X_h})(\X_g)$. Combining the estimate with \eqref{eq:assintotica de mu} then implies that
\begin{align}\label{eq:assintotica de mu vs nu} R_{\X_h}(\mu)(\X_g) \asymp \mu(\X_g)  \asymp \frac{\mu(\X_h)}{\nu(\X_h)} R_{\X_h}(\nu)(\X_g) \leq  \frac{\mu(\X_h)}{\nu(\X_h)} \nu(\X_g) ,\end{align}
provided that $h$ is sufficiently close to $[g,\sigma)$. Now assume that $w \in \cW^n$ for some $n \in \N$ and $g \in G$. It follows from conformality as in \ref{eq:comparability-of-boundary-measure} that $\nu([w,g]) \asymp \varphi_w r^{-n} \nu(\X_{g\kappa^n(w)})$. However, by the choice of $x$, there exists $K(g,w) \in \N$ such that $h= \kappa^{n_k}(x)$ is  sufficiently close to $g\kappa^n(w),\sigma)$ for all $k \geq K(g,w)$. This proves that
\begin{align}\label{eq:assintotica de mu vs nu 2} \mu([w,h]) \ll \frac{\mu(\X_{\kappa^{n_k}(x)})}{\nu(\X_{\kappa^{n_k}(x)})}  \nu([w,h]) \quad \forall  k \geq K(g,w).\end{align}

\medskip
\noindent\textsc{Step 3. Minimality.} Assume that $\nu \leq \mu$. In order to show that $\nu$ is proportional to $\mu$, set  $b := \hbox{ess~inf } {d\nu}/{d\mu}$ and $a := \hbox{ess~sup } {d(\nu - b\mu)}/{d\mu}$. If $a =0$, then $\nu = b\mu$ and there is nothing left to show.  If $a> 0$, consider ${\nu}_1 := a^{-1} (\nu - b\mu)$. Then
\begin{align} \label{eq:ess-inf of nu} \hbox{ess~inf } \frac{d{\nu}_1}{d\mu} = a \left( \hbox{ess~inf } \frac{d\nu}{d\mu} - b\right) =0, \quad
\hbox{ess~sup } \frac{d{\nu}_1}{d\mu} = a^{-1} \left( \hbox{ess~sup } \frac{d\nu}{d\mu} - b\right) =1.
\end{align}
Moreover, it follows from construction that $\nu_2 := \mu - \nu_1$ has the same property.
Hence, for each $\epsilon  > 0$, there exists $A$ of positive measure such that $\nu_2(A) < \epsilon \mu(A)$.
Through approximation by cylinder sets, we may suppose without loss of generality that $A$ is a cylinder set.
It follows from  \eqref{eq:assintotica de mu vs nu 2} for $k$ sufficiently large that
\[  \epsilon > \frac{\nu_2(A)}{\mu(A)} \asymp \frac{\nu_2(\X_{\kappa^{n_k}(x)})}{\mu(\X_{\kappa^{n_k}(x)})}.\]
Hence, as $\epsilon$ is arbitrary,
\[1 \geq  \liminf_{k\to \infty} \frac{\nu_1\left(\X_{\kappa^{n_k}(x)}\right)}{\mu\left(\X_{\kappa^{n_k}(x)}\right)} =
1 - \limsup_{k\to \infty} \frac{\nu_2\left(\X_{\kappa^{n_k}(x)}\right)}{\mu\left(\X_{\kappa^{n_k}(x)}\right)} = 1.
\]
Equation \eqref{eq:assintotica de mu vs nu 2} now implies that $\mu \ll \nu_1$, contradicting \eqref{eq:ess-inf of nu}. Hence $a=0$ and $\nu = b\mu$.

\medskip
\noindent\textsc{Step 4. Existence of the limit.} Assume that $\tilde{\mu}$ is given by a converging subsequence with respect to some arbitrary $(x,g)$ in the equivalence class of $\sigma$.
It follows from \eqref{eq:comparability-of-boundary-measure} that  $d\tilde{\mu}/d\mu \leq C$ for some $C>0$.Then $C^{-1}\tilde{\mu} \leq \mu$. It follows from minimality that  ${\mu}$ and $\tilde{\mu}$ are colinear. As $\mu(\X_g) = \tilde{\mu} (\X_g)$, it follows that $\tilde{\mu} = {\mu}$.

\medskip
\noindent\textsc{Step 5. Unicity.} Assume that $\mu_\sigma(\X_g) \asymp \mu_{\tilde{\sigma}}(\X_g)$ with respect to a constant which is independent from $g \in G$. Furthermore, assume that $h$ and $\tilde{h}$ are elements of the geodesic from $\sigma$ to $\tilde{\sigma})$. By choosing $h$ and $\tilde{h}$ sufficiently distant from each other, it follows that for each $g \in G$, either $h$ is sufficiently close to $[g,\sigma)$  or $\tilde{h}$ is sufficiently close to $[g,\tilde{\sigma})$. Hence, by \eqref{eq:assintotica de mu vs nu} applied simultaneously to  $\mu_\sigma$ and $\mu_{\tilde{\sigma}}$,
\begin{align}\label{eq:upper bound for boundary measure} \mu_\sigma(\X_g) \asymp \mu_{\tilde{\sigma}}(\X_g) \ll R_{\X_h}(\mu_\sigma)(\X_g) + R_{\X_{\tilde{h}}}(\mu_{\tilde{\sigma}})(\X_g) =: \nu(\X_g).\end{align}
 Moreover,
as $\nu$ is excessive but not conformal, it follows for $w \in \cW^n$ that
\begin{align*}
\nu([w,g])  = \int \1_{[w,g]}d\nu \geq r^n \int \1_{[w,g]}d(\cL^n)^\ast(\nu) =  r^n \int \cL^n(\1_{[w,g]})d\nu
 \asymp r^n  \varphi_w \nu(T^n([w,g])).
\end{align*}
By repeating the argument for a finite collections of disjoint words $(u_i)$ contained in $\theta^n([w])$ such that $\kappa^{|u_i|}(u_i) = \id$ and $\bigcup_i \theta^{|u_i|}([u_i]) = \Sigma$,
\begin{align*}
\nu([w,g]) & \gg r^n  \varphi_w \nu(T^n([w,g]))   \geq r^n  \varphi_w \sum_i \nu([u_i,g\kappa^n(w)])\\
& \gg r^n  \varphi_w \sum_i r^{|u_i|}  \varphi_{u_i} \nu(\X_{g\kappa^n(w)}) \gg r^n  \varphi_w \nu(\X_{g\kappa^n(w)}).
\end{align*}
Hence, by combining this estimate with \eqref{eq:comparability-of-boundary-measure} and \eqref{eq:upper bound for boundary measure}, there exists $C>0$ such that
$ \mu_\sigma([w,g]) \leq  C\nu([w,g])$ for all $w \in \cW^n$, $n \in \N$ and $g \in G$. Hence, $ \mu_\sigma \leq C \nu$. However,
as $\nu$ is a potential, that is, it can be written as $\nu = \G_r^\ast(m)$, the Riesz decomposition implies that $\mu_\sigma = 0$, which is a contradiction.

\medskip
\noindent\textsc{Step 6. Limits of ${\mu_\sigma(\X_g)}$, $\G_r(\X_g)$ and  ${\mu_\sigma(\X_g)}/{\mu_{\tilde{\sigma}}(\X_g)}$.} As above, assume that $\sigma$ and $\tilde{\sigma}$ are in $\partial G$, and that $h$ and $\tilde{h}$ are elements of a geodesic from $\sigma$ to $\tilde{\sigma}$ passes first through $h$ and then through $\tilde{h}$. Then, by \eqref{eq:assintotica de mu vs nu},
\begin{align*}
 \frac{\mu_\sigma(\X_{\tilde{h}})}{\mu_{\tilde{\sigma}}(\X_{\tilde{h}})} \ll  \frac{\mu_\sigma(\X_g)}{\mu_{\tilde{\sigma}}(\X_g)} \ll
 \frac{\mu_\sigma(\X_{{h}})}{\mu_{\tilde{\sigma}}(\X_{{h}})},
\end{align*}
for all $g$ such that the geodesic rays $[g,\sigma)$ and  $[g,\tilde\sigma)$ are sufficiently close to $h$ and $\tilde{h}$, respectively. This is illustrated in figure \ref{fig:bi-estimate} for the case of $G$ acting isometrically on the Poincaré disc. In there, the grey part stands for the possible locations of $g$.
\begin{figure}[h] 
   \centering
      \def\svgwidth{0.6\textwidth}
\begingroup%
  \makeatletter%
  \providecommand\color[2][]{%
    \errmessage{(Inkscape) Color is used for the text in Inkscape, but the package 'color.sty' is not loaded}%
    \renewcommand\color[2][]{}%
  }%
  \providecommand\transparent[1]{%
    \errmessage{(Inkscape) Transparency is used (non-zero) for the text in Inkscape, but the package 'transparent.sty' is not loaded}%
    \renewcommand\transparent[1]{}%
  }%
  \providecommand\rotatebox[2]{#2}%
  \newcommand*\fsize{\dimexpr\f@size pt\relax}%
  \newcommand*\lineheight[1]{\fontsize{\fsize}{#1\fsize}\selectfont}%
  \ifx\svgwidth\undefined%
    \setlength{\unitlength}{547.18310061bp}%
    \ifx\svgscale\undefined%
      \relax%
    \else%
      \setlength{\unitlength}{\unitlength * \real{\svgscale}}%
    \fi%
  \else%
    \setlength{\unitlength}{\svgwidth}%
  \fi%
  \global\let\svgwidth\undefined%
  \global\let\svgscale\undefined%
  \makeatother%
  \begin{picture}(1,0.69766121)%
    \lineheight{1}%
    \setlength\tabcolsep{0pt}%
    \put(0,0){\includegraphics[width=\unitlength,page=1]{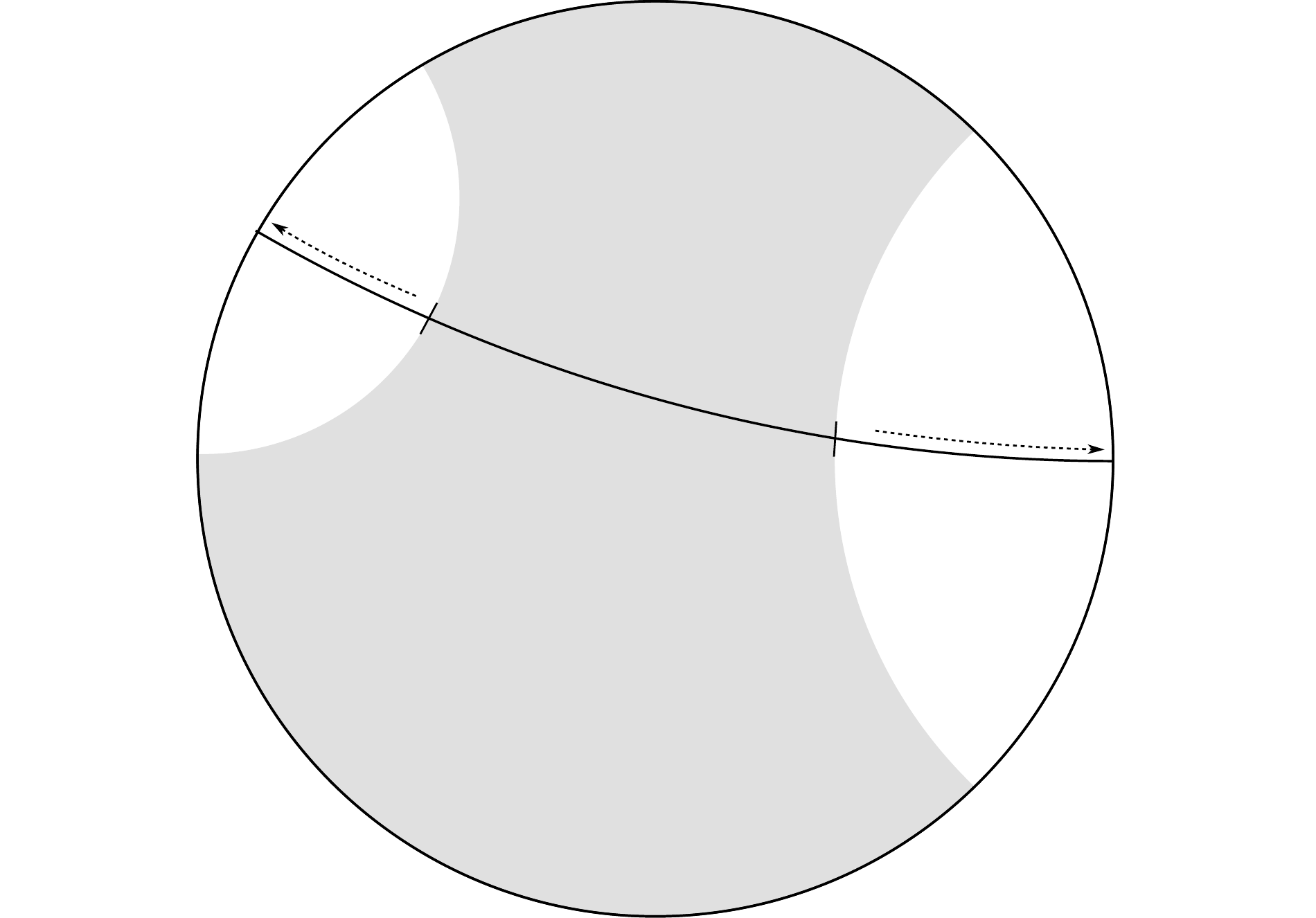}}%
    \put(0.15347495,0.52190663){\color[rgb]{0,0,0}\makebox(0,0)[lt]{\lineheight{1.25}\smash{\begin{tabular}[t]{l}$\tilde{\sigma}$\end{tabular}}}}%
    \put(0.24227803,0.52190663){\color[rgb]{0,0,0}\makebox(0,0)[lt]{\lineheight{1.25}\smash{\begin{tabular}[t]{l}$\tilde{\gamma}$\end{tabular}}}}%
    \put(0.27702709,0.42731202){\color[rgb]{0,0,0}\makebox(0,0)[lt]{\lineheight{1.25}\smash{\begin{tabular}[t]{l}$\tilde{h}$\end{tabular}}}}%
    \put(0.85776069,0.3383545){\color[rgb]{0,0,0}\makebox(0,0)[lt]{\lineheight{1.25}\smash{\begin{tabular}[t]{l}${\sigma}$\end{tabular}}}}%
    \put(0.78957536,0.38194521){\color[rgb]{0,0,0}\makebox(0,0)[lt]{\lineheight{1.25}\smash{\begin{tabular}[t]{l}${\gamma}$\end{tabular}}}}%
    \put(0.65154446,0.31727343){\color[rgb]{0,0,0}\makebox(0,0)[lt]{\lineheight{1.25}\smash{\begin{tabular}[t]{l}${h}$\end{tabular}}}}%
    \put(0.3735522,0.20916532){\color[rgb]{0,0,0}\makebox(0,0)[lt]{\lineheight{1.25}\smash{\begin{tabular}[t]{l}${g}$\end{tabular}}}}%
    \put(0,0){\includegraphics[width=\unitlength,page=2]{estimate-measures-tex.pdf}}%
  \end{picture}%
\endgroup%
   \caption{The positions of $\sigma,\tilde{\sigma}$ and $h,\tilde{h}$. }
   \label{fig:bi-estimate}
\end{figure}
Moreover,
 for $\gamma$ such that $[\gamma,\tilde\sigma)$ passes sufficiently close to ${h}$, the same argument shows that
${\mu_\sigma(\X_{g})}/{\mu_{\tilde{\sigma}}(\X_{g})} \gg   {\mu_{\sigma}(\X_{h})}/{\mu_{\tilde{\sigma}}(\X_{h})}$.
As $\mu_\sigma(\X_{\gamma})  \asymp \mu_{\tilde{\sigma}}(\X_{\gamma})$ for $\gamma$ in a subsequence converging to $\sigma$  would imply that $\mu_\sigma =  \mu_{\tilde{\sigma}}$, it follows
that $\lim_{\gamma \to {\sigma} } {\mu_\sigma(\X_{\gamma})}/{\mu_{\tilde{\sigma}}(\X_{\gamma})}= \infty$. By repeating the argument for $\tilde\gamma \to \tilde\sigma$, one obtains that
\begin{align*}\label{eq:comparability of boundary measures}
 0 \xleftarrow{\tilde\gamma \to \tilde\sigma}  \frac{\mu_\sigma(\X_{{\tilde\gamma}})}{\mu_{\tilde{\sigma}}(\X_{{\tilde\gamma}})}  \ll  \frac{\mu_\sigma(\X_{\tilde{h}})}{\mu_{\tilde{\sigma}}(\X_{\tilde{h}})} \ll  \frac{\mu_\sigma(\X_g)}{\mu_{\tilde{\sigma}}(\X_g)} \ll
 \frac{\mu_\sigma(\X_{{h}})}{\mu_{\tilde{\sigma}}(\X_{{h}})} \ll \frac{\mu_\sigma(\X_{{\gamma}})}{\mu_{\tilde{\sigma}}(\X_{{\gamma}})} \xrightarrow{\gamma \to \sigma} \infty.
\end{align*}
If, in addition, $g$ is an element of a geodesic from $\sigma$ to $\tilde{\sigma}$, then \eqref{eq:ancona-for-conformal} and symmetry imply that
\[ \frac{\mu_\sigma(\X_{{\gamma}})}{\mu_{\tilde{\sigma}}(\X_{{\gamma}})}
\asymp  \frac{ \G_r\left(\X_{g}\right)\left(y, \gamma \right)^{-1} \mu_\sigma(\X_{{g}})}{ \G_r\left(\X_{\gamma}\right)\left(y, g \right) \mu_{\tilde{\sigma}}(\X_{{g}})}   \asymp
  \G_r\left(\X_{g}\right)\left(y, \gamma \right)^{-2} \frac{ \mu_\sigma(\X_{{g}})}{  \mu_{\tilde{\sigma}}(\X_{{g}})} .  \]
Hence, $\G_r\left(\X_{g}\right)\left(y, \gamma \right) \to 0$ as ${\gamma \to \sigma}$ and, by compactness of $\overline{G}$, $\G_r\left(\X_{g}\right)\left(y, \gamma \right) \to 0$ uniformly as $|\gamma| \to \infty$. Therefore, a further application of part (i) of Theorem \ref{theo:Ancona-Gouezel inequalities} implies that the convergence is exponential, that is
\[ \limsup_{n \to \infty} \max_{y \in \Sigma, |\gamma|=n} \sqrt[n]{\G_r\left(\X_{g}\right)\left(y, \gamma \right)} < 1.\]
Furthermore, it follows from \eqref{eq:ancona-for-conformal} that $\lim_{\gamma \to \sigma}  \mu_\sigma(\X_{\gamma}) = \infty $ and $\lim_{\gamma \to \sigma}  \mu_{\tilde{\sigma}}(\X_{\gamma}) = 0$ if $\tilde{\sigma} \neq {\sigma}$.
In order to analyse the behaviour of ${ \mu_\sigma(\X_{{g}})}/{  \mu_{\tilde{\sigma}}(\X_{{g}})}$ for $g$ distant from  the geodesic from $\sigma$ to $\tilde{\sigma}$
fix  $g_n \to g_\infty \in \partial G \setminus \{\sigma,\tilde{\sigma}\}$. Then part (ii) of Theorem \ref{theo:Ancona-Gouezel inequalities} implies as in Proposition \ref{prop:Xi_is_well_defined} that $ \log ({\mu_\sigma(\X_{g_n})}/{\mu_{\tilde{\sigma}}(\X_{g_n})})$ is a Cauchy sequence and that the function $g \to \log ({\mu_\sigma(\X_{g})}/{\mu_{\tilde{\sigma}}(\X_{g})})$ extends continuously to $ \overline{G} \setminus \left\{\sigma,\tilde{\sigma}\right\}$. The remaining assertion is an immediate corollary of part (ii) of Theorem \ref{theo:Ancona-Gouezel inequalities}.

\medskip
\noindent\textsc{Step 7. Integral representation} Now assume that $\mu$ is a minimal conformal measure. It then follows from
 Corollary 3.9 in \cite{Shwartz:2019a} that $\mu$ can be represented by $\int f d\mu = c \mathbb{K}_r(f,\omega)$, for some $c > 0$ and  $\omega \in \mathcal{M}_r$. Now let $(x,g)$ be such that $T^n(x,g) \xrightarrow{n\to \infty} \omega$  in $\mathcal{M}_r$. Moreover, let $\sigma \in \partial G$ be an accumulation point of $(g\kappa_n(x))$. Then, by the first assertion in (i), $\mu = c \mu_\sigma$. In particular,  $\partial G$ can be identified with the set of minimal conformal measures, which proves the second assertion of (i). The representation of arbitrary conformal measures then is a corollary of Theorem 3.12 in \cite{Shwartz:2019a}.
\end{proof}

Now assume that $z \in M_r$ and define $m_z(\X_h) := \lim_{n\to \infty} \mathbb{K}_r(h,T^n(z))$. It now follows as in the above proof of (i) of Theorem \ref{theo:geometric-boundary} that $m_z$ extends uniquely to a minimal conformal measure. Therefore, there exists a unique  $\sigma \in \partial G$ such that  $\mu_\sigma = m_z$, which proves the following.
\begin{corollary} The map $\Xi: \partial G \to \mathcal{M}_r$ is a bijection.
\end{corollary}

In order to analyse the topological properties of $\Xi$, we start with the construction of a metric which is compatible with the Martin compactification of $\Sigma \times G$. However, in order to obtain Hölder continuity, it will turn out that we have to modify the classical definition slightly by taking logarithms as follows. By Lemma \ref{lem:bounded distortion for the extended operator},  $\|\log \mathbb{K}_r(h,\cdot)\|_\infty < \infty$. Hence, there exists $\{c_h > 0: h \in G\}$
such that
\[ \Delta_r((x,g), (\tilde x,\tilde g)) := \sum_{h \in G} c_h \left|\log \mathbb{K}_r(h,(x,g)) - \log \mathbb{K}_r(h,(\tilde x,\tilde g))\right|  \ll 1 \]
for all $(x,g), (\tilde x,\tilde g) \in \Sigma\times G$. Furthermore, if $(\tilde x,\tilde g) \in M_r$ then
\[d^r_{\hbox{\tiny Martin}} ((x,g), (\tilde x,\tilde g)) :=
\lim_{n \to \infty} \Delta_r(T^n(x,g), T^n(\tilde x,\tilde g)), \quad \hbox{ for } (x,g), (\tilde x,\tilde g) \in M_r
\]
is well defined and, as it easily can be shown, defines a metric on $\mathcal{M}_r$. It is worth noting that it follows from general topology that
$\Sigma \times G \cup \mathcal{M}_r$ is the unique compactification of $\Sigma \times G$ such that each $\mathbb{K}_r(h,\cdot)$ extends to a continuous function, and that $d^r_{\hbox{\tiny Martin}}$ is a metric for this topology (see, e.g., \cite{Woess:2009}). In particular,
 this topology is independent from the parameters $\{c_h\}$ and from considering $\log \mathbb{K}_r$ instead of $\mathbb{K}_r$. However, Theorems \ref{theo:Ancona-Gouezel inequalities} and  \ref{theo:geometric-boundary} allow to obtain precise estimates for $|\log \mathbb{K}_r(h,\cdot)|$ and $|\log \mathbb{K}_r(h,\cdot) - \log \mathbb{K}_r(h,\cdot\cdot)|$ which gives rise to the following definition of $c_h$, for $h \in G$ and $\lambda$ as in Theorem \ref{theo:Ancona-Gouezel inequalities}.
\begin{equation} \label{def:coefficients}
 c_h := \frac{{\lambda}^{2|h|}}{ \#\left\{ g \in G: |g| = |h|\right\} \left|\log \mathbb{G}_r(\X_{h})(x,\id))\right| }.
 \end{equation}
 Furthermore, observe that the sequence $(\log \#\left\{ g \in G: |g| = n \right\})$ is sub-additive, which implies that the following exponential growth rate $\mathfrak{h}$ exists.
 \[ \mathfrak{h} := \lim_{n \to \infty} \sqrt[n]{\#\left\{ g \in G: |g| = n \right\}}.\]

\begin{theorem} \label{theo:continuidade} Assume that $G$ is hyperbolic, $T$ is a topologically transitive extension of a Gibbs-Markov map of finite type and that $\G_R(\X_g)(x,\id) \asymp \G_R(\X_{\id})(x,g)$, independent of  $(x,g)\in \Sigma \times G$ and that  $r \leq R$.
Then, the map  $\Xi: (\partial G, d_{\hbox{\tiny visual}} ) \to (\mathcal{M}_r, d^r_{\hbox{\tiny Martin}})$ is a homeomorphism and, if $d^r_{\hbox{\tiny Martin}}$ is defined through $(c_h)$ as in \eqref{def:coefficients}, then
\[  d_{\hbox{\tiny visual}}(\sigma,\tilde \sigma)^\beta \ll d^r_{\hbox{\tiny Martin}}(\Xi(\sigma),\Xi(\tilde \sigma))  \ll d_{\hbox{\tiny visual}}(\sigma,\tilde \sigma)^\alpha , \]
for $\alpha = \log \lambda/\log \lambda_{\hbox{\tiny visual}}$, $\beta = (2\log \lambda - \log(\mathfrak{h}+ \epsilon))/\log \lambda_{\hbox{\tiny visual}}$ and an arbitrary $\epsilon > 0$. In particular,  $\Xi$ and $\Xi^{-1}$ are Hölder continuous with exponents $\alpha$ and $1/\beta$, respectively.
\end{theorem}

\begin{proof}
In order to deduce Hölder continuity, we begin with a geometric description of the visual metric. Suppose that
$\sigma,\tilde \sigma \in \partial G$. By Proposition \ref{prop:Xi_is_well_defined}, there exist $x,\tilde x \in \Sigma$ such that $\lim \kappa^n(x) = \sigma$ and $\lim \kappa^n(\tilde x) = \tilde \sigma$.
Furthermore, it follows from the above estimates for the visual metric that
\begin{align*}
 d_{\hbox{\tiny visual}}(\sigma,\tilde \sigma)  & \geq C \lambda_{\hbox{\tiny visual}}^{(\sigma \cdot \tilde \sigma )} \geq  C \lambda_{\hbox{\tiny visual}}^{\liminf_{m,n \to \infty} (\kappa^m(x)\cdot \kappa^n(\tilde x)) + 2 \delta} \\
 & = \left(C\lambda^{2\delta} \right) \lambda_{\hbox{\tiny visual}}^{\liminf_{m,n \to \infty} (\kappa^m(x)\cdot \kappa^n(\tilde x)) }.
\end{align*}
Moreover, as $(\kappa^n(x))$ and  $(\kappa^n(\tilde x))$ converge at infinity, it also follows that these sequences leave any finite subset of $G$. Now suppose that $\liminf_{m,n \to \infty} (\kappa^m(x)\cdot \kappa^n(\tilde x)) = N$. Then there exists $m,n$ arbitrary large such that $ (g\cdot \tilde g) = N$, for $g := \kappa^m(x)$ and $\tilde g := \kappa^n(\tilde x)$. In particular,  $ d_{\hbox{\tiny visual}}(\sigma,\tilde \sigma) \asymp \lambda_{\hbox{\tiny visual}}^N$.

This geometric characterization now allows to employ Theorems \ref{theo:Ancona-Gouezel inequalities} and Theorem \ref{theo:geometric-boundary} in  order to obtain refined estimates for $|\log \mathbb{K}_r(h,\cdot)|$ and $|\log \mathbb{K}_r(h,\cdot) - \log \mathbb{K}_r(h,\cdot\cdot)|$. For $g, \tilde g, \in G$ and let $\xi \in G$ refer to the element in the geodesic arc $[\id, h]$ closest to $g$, and
$\xi_1,\xi_2 \in G$ refer the elements in a geodesic arc $[g, \tilde g]$ closest to $h$ and $\id$, respectively (see Figure \ref{fig:Configurations}).
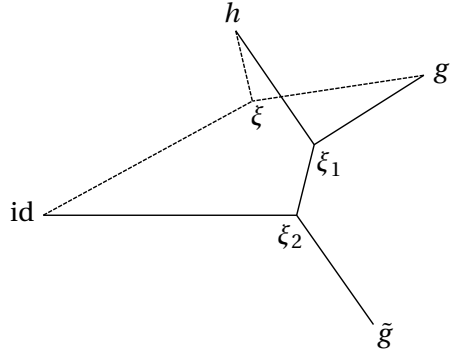
\begin{figure}[htbp] 
   \centering
         \def\svgwidth{0.4\textwidth}
 \begingroup%
  \makeatletter%
  \providecommand\color[2][]{%
    \errmessage{(Inkscape) Color is used for the text in Inkscape, but the package 'color.sty' is not loaded}%
    \renewcommand\color[2][]{}%
  }%
  \providecommand\transparent[1]{%
    \errmessage{(Inkscape) Transparency is used (non-zero) for the text in Inkscape, but the package 'transparent.sty' is not loaded}%
    \renewcommand\transparent[1]{}%
  }%
  \providecommand\rotatebox[2]{#2}%
  \newcommand*\fsize{\dimexpr\f@size pt\relax}%
  \newcommand*\lineheight[1]{\fontsize{\fsize}{#1\fsize}\selectfont}%
  \ifx\svgwidth\undefined%
    \setlength{\unitlength}{377.00360619bp}%
    \ifx\svgscale\undefined%
      \relax%
    \else%
      \setlength{\unitlength}{\unitlength * \real{\svgscale}}%
    \fi%
  \else%
    \setlength{\unitlength}{\svgwidth}%
  \fi%
  \global\let\svgwidth\undefined%
  \global\let\svgscale\undefined%
  \makeatother%
  \begin{picture}(1,0.7528282)%
    \lineheight{1}%
    \setlength\tabcolsep{0pt}%
    \put(0,0){\includegraphics[width=\unitlength,page=1]{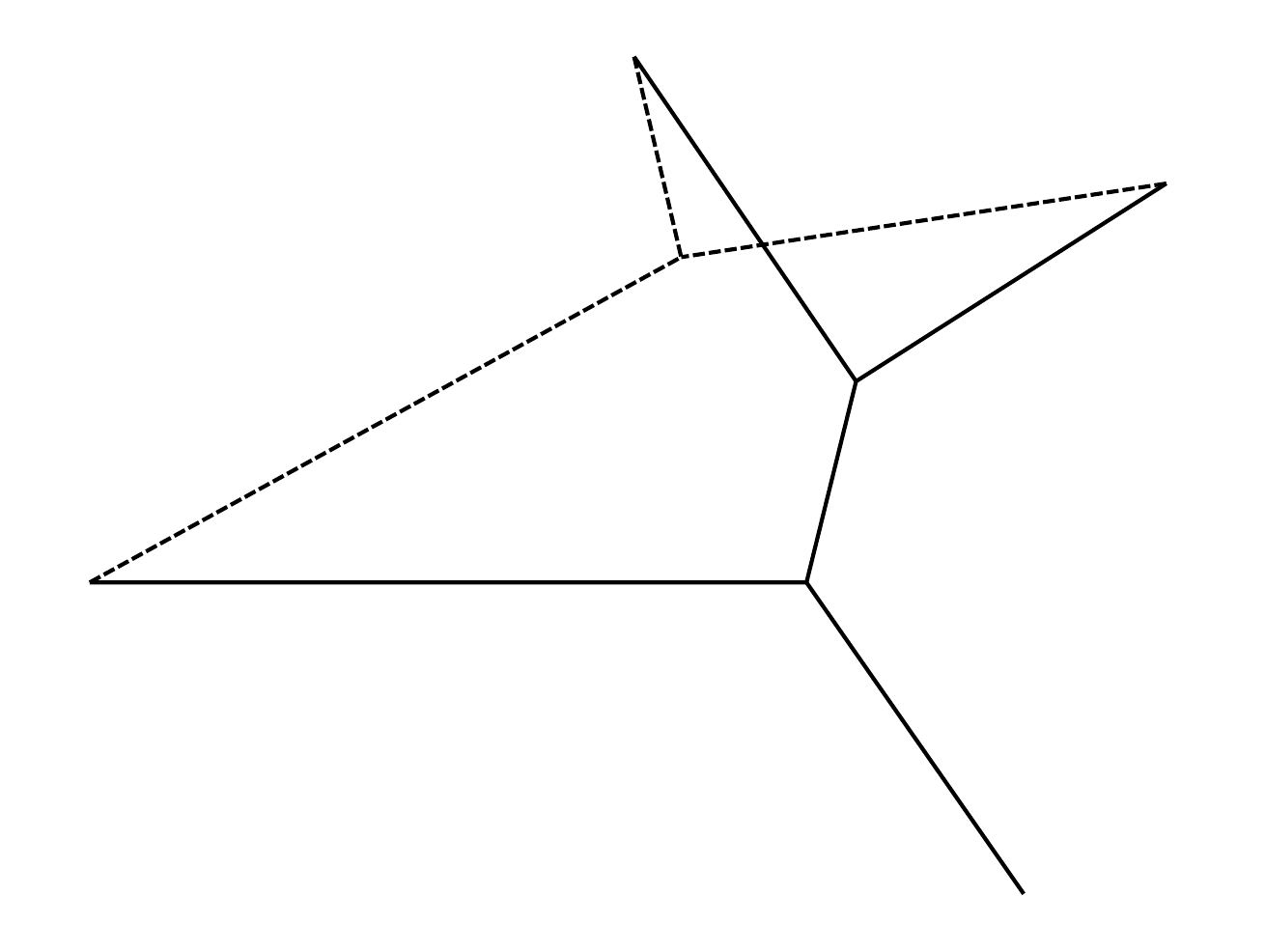}}%
    \put(-0.0028105,0.27953743){\color[rgb]{0,0,0}\makebox(0,0)[lt]{\lineheight{1.25}\smash{\begin{tabular}[t]{l}$\id$\end{tabular}}}}%
    \put(0.59520997,0.22767443){\color[rgb]{0,0,0}\makebox(0,0)[lt]{\lineheight{1.25}\smash{\begin{tabular}[t]{l}$\xi_2$\end{tabular}}}}%
    \put(0.68255812,0.39145224){\color[rgb]{0,0,0}\makebox(0,0)[lt]{\lineheight{1.25}\smash{\begin{tabular}[t]{l}$\xi_1$\end{tabular}}}}%
    \put(0.81631003,0.00657425){\color[rgb]{0,0,0}\makebox(0,0)[lt]{\lineheight{1.25}\smash{\begin{tabular}[t]{l}$\tilde{g}$\end{tabular}}}}%
    \put(0.94187307,0.60436336){\color[rgb]{0,0,0}\makebox(0,0)[lt]{\lineheight{1.25}\smash{\begin{tabular}[t]{l}$g$\end{tabular}}}}%
    \put(0.47340025,0.72833407){\color[rgb]{0,0,0}\makebox(0,0)[lt]{\lineheight{1.25}\smash{\begin{tabular}[t]{l}$h$\end{tabular}}}}%
    \put(0.53527189,0.49631553){\color[rgb]{0,0,0}\makebox(0,0)[lt]{\lineheight{1.25}\smash{\begin{tabular}[t]{l}$\xi$\end{tabular}}}}%
  \end{picture}%
\endgroup%
   \caption{A configuration of $g, \tilde g$ and  $h,\id$  with $\xi_1 \neq \xi_2$}
   \label{fig:Configurations}
\end{figure}
\begin{enumerate}
\item By applying part (i) of  Theorem \ref{theo:Ancona-Gouezel inequalities}, it follows that, for some uniform constant $C>0$ ,
\[\left|\log \mathbb{K}_r(h,(x,g)) \right|  =  \left|\log \mathbb{K}_r(h,(x,\xi))  \right|  \pm C. \]
It follows from the symmetry $\G_r(\X_g)(x,\id) \asymp \G_r(\X_{\id})(x,g)$ and  a further application of  part (i) of  Theorem \ref{theo:Ancona-Gouezel inequalities} that
$\mathbb{K}_r(h,(x,\xi)) \asymp  \mathbb{G}_r(\X_{h})(x,\id)/\mathbb{G}_r(\X_{\xi})(x,\id)^2$.
As $\mathbb{G}_r(\X_{h})(x,\id) \gg \mathbb{G}_r(\X_{\xi})(x,\id)$, we obtain that
\[\left|\log \mathbb{K}_r(h,(x,g)) \right|  =
\left|\log \mathbb{G}_r(\X_{h})(x,\id)) - 2\log \mathbb{G}_r(\X_{\xi})(x,\id))   \right|
\pm C \leq  \left|\log \mathbb{G}_r(\X_{h})(x,\id))\right| + C. \]
Moreover, part (v) of Theorem \ref{theo:geometric-boundary} implies that this bound grows linearly in $|h|$.
\item If $g, \tilde g$ and  $h,\id$ are in configuration as in part (ii) of Theorem \ref{theo:Ancona-Gouezel inequalities}, then  $\xi_1=\xi_2$ and
\[\left|\log \mathbb{K}_r(h,(x,g)) - \log \mathbb{K}_r(h,(\tilde x,\tilde g))\right|  \ll \lambda^{(g \cdot \tilde g)-|h|}.\]
\item If $g, \tilde g$ and  $h,\id$ are not necessarily in this configuration (see Figure \ref{fig:Configurations}), then part (i) above implies that
\[ \left|\log \mathbb{K}_r(h,(x,g)) - \log \mathbb{K}_r(h,(\tilde x,\tilde g))\right| = 2 \left|\log \mathbb{G}_r(\X_{\xi_1})(x,\xi_2))\right| \pm C.\]
\end{enumerate}
We begin with the estimate from above. In order to do so, observe that $|\log \mathbb{G}_r(h,(x,\id))|^{-1}$ tends to $0$ as $|h| \to \infty$ by (v) of Theorem \ref{theo:geometric-boundary}.
This implies that
 $\sum_{h \in G} c_h < \infty$. Moreover, by approximation by a graph, we may assume that
 $ |\xi_2| = (g \cdot \tilde g) =: N$. Furthermore, we have that $|h| \leq N/2$ implies that  $h,\id$ and $g,\tilde g$ are in a configuration as in (ii) of Theorem \ref{theo:Ancona-Gouezel inequalities}. Hence
\begin{align*}
d^r_{\hbox{\tiny Martin}}(\sigma,\tilde \sigma) &  = \lim_{g \to \sigma, \tilde g \to \tilde \sigma}\;
\sum_{|h| \leq N/2} c_h \left|\log \mathbb{K}_r(h,(x,g)) - \log \mathbb{K}_r(h,(\tilde x,\tilde g))\right|  \\
& \phantom{=} + \lim_{g \to \sigma, \tilde g \to \tilde \sigma}\;
 \sum_{|h| > N/2} c_h \left|\log \mathbb{K}_r(h,(x,g)) - \log \mathbb{K}_r(h,(\tilde x,\tilde g))\right| \\
& \ll  \sum_{n = 1}^{N/2}   \sum_{|h|=n } c_h  {\lambda}^{(g \cdot \tilde g)-|h|} +  \sum_{|h| > N/2} c_h
\ll
 {\lambda}^{N} \sum_{n = 1}^{\infty}  {\lambda}^{n} +   \sum_{n = N/2}^{\infty}  {\lambda}^{2n} \ll  {\lambda}^{N}.
\end{align*}
For the estimate from below, it suffices to consider  the term with $h=\xi_1=\xi_2$. Namely, it follows from (iii) that
\begin{align*}
d^r_{\hbox{\tiny Martin}}(\sigma,\tilde \sigma) &  \geq \lim_{g \to \sigma, \tilde g \to \tilde \sigma}\;
  c_{|\xi_2|} \left|\log \mathbb{K}_r(\xi_2,(x,g)) - \log \mathbb{K}_r(\xi_2,(\tilde x,\tilde g))\right|\\
 &  \gg \frac{{\lambda}^{2N}}{\kappa_{N }}
 {\left|\log \mathbb{G}_r(\X_{\xi_2})(x,\id))\right| }^{-1}.
\end{align*}
As $\left|\log \mathbb{G}_r(\X_{\xi_2})(x,\id))\right|$ grows linearly, it follows that $d^r_{\hbox{\tiny Martin}}(\sigma,\tilde \sigma)  \geq C_\epsilon (\lambda^2/(\mathfrak{h}+ \epsilon))^N$, for any $\epsilon > 0$ and a constant $C_\epsilon > 0$. Hence, we have shown that
\[  d_{\hbox{\tiny visual}}(\sigma,\tilde \sigma)^\beta \asymp \lambda_{\hbox{\tiny visual}}^{\beta N}  = \frac{\lambda^{2N}}{(\mathfrak{h}+ \epsilon)^N}  \ll d^r_{\hbox{\tiny Martin}}(\Xi(\sigma),\Xi(\tilde \sigma))  \ll \lambda^N = \lambda_{\hbox{\tiny visual}}^{\alpha N} \asymp d_{\hbox{\tiny visual}}(\sigma,\tilde \sigma)^\alpha , \]
for $\alpha = \log \lambda/\log \lambda_{\hbox{\tiny visual}}$ and $\beta = (2\log \lambda - \log(\mathfrak{h}+ \epsilon))/\log \lambda_{\hbox{\tiny visual}}$.
\end{proof}

\section{An application to regular covers of convex-cocompact CAT(-1) metric spaces}\label{sec:regular cover}

We give an application of our main theorem in order to caracterize the set of $\delta$-conformal measures of a regular cover of a  convex-cocompact CAT(-1) metric space, where we assume that $\delta$ is the abcissa of convergence of the cover and that the covering group is word hyperbolic. 

We now recall the basic definitions and refer for the details to \cite{DasSimmonsUrbanski:2017}. A \emph{CAT(-1) space} $X$ is a geodesic space such that each geodesic triangle is thinner than a comparison triangle in the hyperbolic plane with constant curvature -1. An important feature of  CAT(-1) spaces is that they are strongly hyperbolic which implies, in particular, that their visual boundary $\partial X$ coincides with the topological boundary, the visual metric on $\partial X$ simplifies to $D_\mathbf{o}(\xi,\eta) := e^{-(\xi,\eta)_\mathbf{o}}$ and, for any isometry $g$ of $X$, the conformal derivative $g'$ exists and satisfies, for any $\xi \in \partial X$
\[ g'(\xi) := \lim_{\eta \to \xi, \eta \in \partial X} \frac{D_\mathbf{o}(g(\xi),g(\eta))}{D_\mathbf{o}(\xi,\eta)} = e^{-B_\xi\left(g^{-1}(\mathbf{o}),\mathbf{o}\right)}.  \]
In here, $B_\xi(x,y) := (y \cdot \xi)_x - (x \cdot \xi)_y$ is the Busemann function, and $\mathbf{o}$ is some point in $X$.

As we are interested in quotients of $X$, assume that  
$\Gamma$ is a discrete subgroup $\Gamma$ of the isometry group $\hbox{Isom}(X)$ of $X$ 
which acts freely and properly discontinuously on $X$.
Furthermore, as it is well known, these assumptions imply that 
$X/\Gamma$ inherits several properties of $X$, but only locally. For example,
each element in $X/\Gamma$ has a neighbourhood which is a CAT(-1) metric space.
A basic object in the analysis of $\Gamma$ is its limit set $\Lambda(\Gamma)$
defined by $\Lambda(\Gamma) := \overline{\Gamma(\mathbf{o})} \cap \partial X$ and
the class of $s$-conformal measures (for $s > 0$), where  a Borel probability measure 
on $\partial X$ is referred to as an $s$-conformal measure if 
\begin{equation}\label{eq:N-conformal measure} \mu(g(A)) = \int_A (g'(\xi))^s d\mu(\xi)\end{equation}    
for any Borel subset $A$ of $\partial X$. The relevance of these measures stems from 
the fact that under reasonable assumptions on $\Gamma$ and $s$, 
there is only one conformal measure, and this measure gives rise to canonical measures for the geodesic flow and the horocyclic foliation. As a standing assumption, we assume from now on that any subgroup of $\hbox{Isom}(X)$ is discrete, and acts freely and properly discontinuously on $X$. To each such $\Gamma$, the  Poincaré exponent is defined by 
\[ \delta(\Gamma) := \sup\left\{s \geq 0: \sum_{g \in \Gamma} e^{-s d(\mathbf{o},g(\mathbf{o}))} = \infty  \right\}.  \]
Moreover, $\Gamma$ is referred to as of divergence type if $\sum_{g \in \Gamma} e^{-\delta d(\mathbf{o},g(\mathbf{o}))} = \infty$. The relevance of this is a consequence of the Hopf-Tsuji theorem (for the setting of CAT(-1) spaces, see \cite{DasSimmonsUrbanski:2017}), which states that $\Gamma$ is of divergence type if and only if, for any $\delta$-conformal measure $\mu$, the  action of $\Gamma$ on $  (\partial X,\mu)^2$ is ergodic. Moreover, being of divergence type implies uniqueness of $\mu$. 

In fact, we are interested in the interplay of the following type of groups. 
\begin{defn}  We refer to a cover $Y$ of $X/\Gamma$ as a regular cover (or periodic cover) if there exists a normal subgroup $N$ of $\Gamma$ such that $Y = X/N$. In this situation, $G:= \Gamma/N$ refers to the covering group (or period) of the cover.  Moreover, we refer to  $\Gamma$ as convex-cocompact if the convex hull of $\Lambda(\Gamma)$ 
in $\overline{X}$ is compact. 
\end{defn}

Observe that the class of convex-cocompact groups acting on a CAT(-1) metric space is a flexible object as manifolds of pinched negative curvature with compact convex core as well as the action of a word hyperbolic group on its Cayley graph are in this class. Furthermore, due to the close connection to the  basic example of an Anosov flow, the geodesic flow on a closed manifold of constant negative curvature, the action of this class of groups is well studied. For example, the  Poincaré exponent of a convex-cocompact group his finite and the group is of divergence type (\cite{DasSimmonsUrbanski:2017}) and, in particular, the geodesic flow is ergodic with respect to the Liouville-Patterson-Sullivan measure constructed from the unique $\delta$-conformal measure.

The geodesic flow on regular covers, on the other hand, in many cases is totally dissipative as the dynamics somehow behave like a random walk on the covering group. That is, even though there does not exist a complete dictionary between random walks and the regular covers, there are several parallel results, like Rees' version of Polya's result on the transience of the simple random walk for  
$\Z^d$-covers (\cite{Rees:1981}), or Brooks' 
amenability criterium (\cite{Brooks:1985}, see also \cite{Stadlbauer:2013,DougallSharp:2015,CoulonDalBoSambusetti:2018}) in the sprit of Kesten (\cite{Kesten:1959a}). 
Our application of Theorem \ref{theo:geometric-boundary} adds a further item to this list as it provides a complete description of the $\delta(N)$-conformal minimal measures in analogy to Ancona's result on the geometric realization of minimal harmonic functions. From now on, we refer to $\mathcal{G}(X)$, $\mathcal{G}(X/\Gamma)$ and $\mathcal{G}(X/N)$ as the space of geodesics of the local CAT(-1) spaces $X$, $X/\Gamma$ and $X/N$.

\begin{theorem}\label{theo:geometric-application} Assume that $X$ is a CAT(-1) space, that $\Gamma$ is a  convex-cocompact, discrete subgroup of $\hbox{Isom}(X)$ and that $N$ is a normal, non-elementary subgroup of $\Gamma$ such that $G:=\Gamma/N$ is  hyperbolic and such that the geodesic flow on $\mathcal{G}(X/N)$ is topologically transitive.  Then the set of minimal, $\delta(N)$-conformal measures can be identified with $\partial G$.       
\end{theorem}

\begin{proof} The strategy of proof is as follows. In Step 1, we construct the Markov map $(\Sigma,\theta)$ and its group extension by $G$. In Step 2, we then specify the associated potential and show that the reversibility condition $\G_R(\X_g)(x,\id) \asymp \G_R(\X_{\id})(x,g)$ holds. Finally, in Steps 3 and 4, we identify the $\delta(N)$-conformal measures on $\partial X$ with the conformal measures on $\Sigma \times G$. The theorem then follows from Theorem \ref{theo:geometric-boundary}.

\medskip
\noindent\textsc{Step 1. The Markov map $\theta$.} 
The first part of proof makes use of coding of the geodesic flow on $\mathcal{G}(X/\Gamma)$  as constructed in \cite{ConstantineLafontThompson:2020}. In there, the authors construct a Poincaré section such that 
\begin{enumerate}
\item the first return map $H$ is 
coded by a topologically transitive bilateral subshift of finite type,
\item \label{coding:atoms-lift}
the atoms  $\left\{A_1, \ldots A_n \right\}$ of the Markov partition of the section are of the form $A_i = \pi(R_i)$, where the $\{R_i\}$ are, in Hopf coordinates,  of the form 
\[R_i = \left\{ (\xi, \eta, t_i(\xi, \eta)) : \xi \in  {U_i}, \eta \in  {V_i} \right \},\]
for some disjoint open subsets $U_i,V_i$ of  $\partial X$ and functions $ t_i(\xi, \eta) :U_i \times V_i \to \R$,
\item the return time $h: \bigcup_{i=1}^n A_i \to (0,\infty)$ to the section is Hölder continuous.     
\end{enumerate}
By considering the $p$-th iterate of $H$, where $p$ refers to the period of $\theta$, the subshift of finite type $\theta^p$ decomposes into its topological mixing components. However, as $H$ is the first return map, any of these components provides us with a Markov coding for the geodesic flow whose associated subshift of finite type is topological mixing. Hence, we may assume without loss of generality that the shift is topologically mixing. 
 
In order to associate elements of $\Gamma$ to $H$, we proceed as follows.  
The possible transitions of $\Sigma$ define a connected graph $\mathfrak{G}$. Now choose a subgraph $\mathfrak{T}$ which is connected, has no loops and has the same set of vertices as $\mathfrak{G}$, that is, $\mathfrak{T}$ is a minimal spanning tree of $\mathfrak{G}$. 
We now construct a lift of the atoms $\left\{A_1, \ldots A_n \right\}$ to $\mathcal{G}(X)$ based on the choice of $\mathfrak{T}$. Suppose that the lift 
$\hat{A}_i \subset \mathcal{G}(X)$ of $A_i$ already was constructed. Furthermore, suppose that ${A}_j$ is a neighbour of $A_i$ in $\mathfrak{T}$. Then $H((A_i)) \cap (A_j) \neq \emptyset$ or  $H(A_j) \cap A_i \neq \emptyset$. In the first case, there exists a unique lift $\hat{A}_j \subset \mathcal{G}(X)$ such that $\left\{g_{h(\pi(x))}(x) : x \in \hat{A}_i \right\} \cap \hat{A}_j \neq \emptyset$. And, if $H(A_i) \cap  A_j  = \emptyset$, the same argument  gives rise to a unique $\hat{A}_i \subset \mathcal{G}(X)$ with  $\left\{g_{h(\pi(x))}(x) : x \in \hat{A}_j \right\} \cap   \hat{A}_i  \neq \emptyset$. As $\mathfrak{T}$ is a minimal spanning tree, this construction provides a construction of lifts $\left\{\hat{A}_1, \ldots \hat{A}_n \right\}$. Moreover, as $\pi \circ g = \pi$ for all $g \in \Gamma$, we may identify $R_i := \hat{A}_i$, for $i=1, \ldots, n$ in the property of the above coding.

 The elements of $\Gamma$ associated to $H$ are now constructed as follows. Suppose that $x = (\xi, \eta, t_i(\xi, \eta)) \in {R_i}$ and that $g_{h(\pi(x))}(\pi(x)) \in A_j$. Then there exists a unique $\kappa_x \in \Gamma$, depending only on $i$ and $j$, such that   
$g_{h(\pi(x))}(x) \in \kappa_x({R}_j)$ and, in Hopf coordinates,
\[ g_{h(\pi(x))}\left(\pi(\xi, \eta, t_i(\xi, \eta))\right) = \pi\left(\xi, \eta, t_i(\xi, \eta)+h(\pi(x)) \right) =  \pi\left(\kappa_x^{-1}(\xi), \kappa_x^{-1}(\eta), t_j(\kappa_x^{-1}(\xi), \kappa_x^{-1}(\eta))  \right).  \]
However, as the functions $t_i$ can be recovered from $\xi$, $\eta$ and $i$, we identify $H$ with  
\begin{equation}\label{eq:two-sided-coding}  \biguplus_{i=1}^n U_i \times V_i \to  \biguplus_{i=1}^n U_i \times V_i, \quad (\xi,\eta) \mapsto \left(\kappa_{(\xi,\eta,t_i(\xi,\eta))}^{-1}(\xi), \kappa_{(\xi,\eta,t_i(\xi,\eta))}^{-1}(\eta)\right), \hbox{ for } (\xi,\eta) \in  U_i \times V_i,  \end{equation}
where $\biguplus$ stands for the disjoint union. The following observation is crucial and follows immediately from the definition of a Markov partition (cf. Def. 3.8 in \cite{ConstantineLafontThompson:2020}). Namely, $\kappa_{(\xi,\eta,t_i(\xi,\eta))}$ in fact only depends on $\eta$ and $i$. Hence, with $\kappa_{\eta,i} := \kappa_{(\xi,\eta,t_i(\xi,\eta))}$, for $\eta \in V_i$, we obtain that 
 
\[
\begin{CD}
H:  \biguplus_{i=1}^n U_i \times V_i @>{(\xi,\eta) \mapsto \left(\kappa_{\eta,i}^{-1}(\xi),\kappa_{\eta,i}^{-1}(\eta) \right) }>>    \biguplus_{i=1}^n U_i \times V_i\\ @VVV @VVV  \\  \theta: \biguplus_{i=1}^n   V_i   @>{\eta \mapsto \kappa_{\eta,i}^{-1}(\eta)}>> \biguplus_{i=1}^n   V_i 
 \end{CD}
\]
commutes. In particular, by setting $ \Sigma:= \biguplus_{i=1}^n   \overline{V}_i$ and $\theta|_{\overline{V}_i} (\eta):=  \kappa_{\eta,i}^{-1}(\eta)$, we obtain a non-invertible, surjective Markov map which is coded by a one-sided topological mixing subshift of finite type and which, up to points on the boundaries $\{\partial V_i\}$, is a factor of $H$. 

\medskip
\noindent\textsc{Step 2. Associated measures and reversibility of the extension.} 
We now analyse the regularity of the potential function 
\begin{equation}  \label{eq:geometric potential} \varphi(\eta,i) := \delta(N) \log  \left((\kappa_{\eta,i}^{-1})'(\eta,i)\right)= - \delta(N)  B_\eta\left(\kappa_{\eta,i}(\mathbf{o}),\mathbf{o}\right) .\end{equation} 
As it is well known, the Busemann function is 1-Lipschitz continuous. Furthermore, it follows from the expansivity of the geodesic flow, that the map $\theta$ is eventually uniformly expanding. By combining these two observations, it follows that $\varphi :\Sigma \to \R$ is Hölder continuous with respect to the shift metric on $\Sigma$. Therefore, by application of Ruelle's operator theorem, there exists a unique equilibrium state $hd\mu$ for $\varphi$, where $h$ is a Hölder continuous function which is bounded away from 0, and $\mu$ is a conformal measure, which means in the setting of Markov maps or shift spaces that $ d\nu\circ\theta = \lambda e^{- \varphi}d\nu$, for some $\lambda > 0$. 

The potential $\varphi$ is related to the Poincaré series of $N$ through the group extension 
\[ T: \Sigma \times G \to \Sigma \times G, (x,gN) \mapsto (\theta(x),g\kappa_xN) \]
as follows.  As $G$ is non-amenable, observe that it follows from the main result in \cite{Stadlbauer:2013} that $\lambda > 1$. In order to determine $R$, we make use of the fact that 
$R$ is the radius of convergence of the series $\G_r(\X_{\id})(x,\id)$, seen as a function of $r$. By conformality and by Lemma \ref{lemma:distortion of L}, 
\begin{align*} \G_r(\X_{\id})(x,\id) \asymp  \int_{\mathcal{X}_{\id}} \G_r(\X_{\id}) hd\mu  = \sum_{n=0}^\infty r^n  \int \X_{\id}\circ T^n  \X_{\id} hd\mu \asymp \sum_{n=0}^\infty r^n  \sum_{w \in \cW^n: \kappa(w) \in N} \mu([w])
\end{align*}
Now set, for $w \in \cW^n$, $\kappa_w := \kappa_x \kappa_{\theta(x)} \cdots \kappa_{\theta^{n-1}(x)}$ for some $x \in [w]$. As the Busemann function is a cocycle, it follows for $x = (\eta,i) \in [w]$ that  
\[S_n(\varphi)(x) :=  \sum_{k=0}^{n-1} \varphi\circ \theta^k(x) =    - \delta(N)  B_\eta\left(\kappa_w(\mathbf{o}),\mathbf{o}\right) . \]     
Furthermore, a well known geometric argument for convex-cocompact groups shows that there exists a constant $K$, independent from $x = (\eta,i)$ and $w$ such that $\eta$ is in  $\{\eta \in \partial X:   (\eta \cdot \mathbf{o}))_{\kappa_w(\mathbf{o})} \leq K \}$, known as the $K$-shadow of $\kappa_w(\mathbf{o})$ from $\mathbf{o}$ of parameter $K$, which then implies that $ B_\eta\left(\kappa_w(\mathbf{o}),\mathbf{o}\right) = d \left(\kappa_w(\mathbf{o}),\mathbf{o}\right) \pm C$, for some $C>0$   (see, e.g., Observation 4.5.3 in \cite{DasSimmonsUrbanski:2017}).   
Hence,  
\begin{align*}  \G_r(\X_{\id})(x,\id) &  \asymp   \sum_{n=0}^\infty (\lambda r)^n  \sum_{w \in \cW^n: \kappa(w) \in N}  e^{ - \delta(N) d\left(\kappa_w(\mathbf{o}),\mathbf{o}\right)}  = 
\sum_{w :  \kappa(w)  \in N}   e^{ - d\left(\kappa_w(\mathbf{o}),\mathbf{o}\right) \left( \delta(N) - \frac{|w| \log (\lambda r)}{d\left(\kappa_w(\mathbf{o}),\mathbf{o}\right) }  \right)}
\end{align*}
As the return time to the section is bounded from away from 0 and infinity by construction, it follows hat  $C^{-1}|w| < d\left(\kappa_w(\mathbf{o}),\mathbf{o}\right)< C |w|$ for some $C> 0$. By combining this estimate with the observation that $\{w : \kappa(w) =g\} \neq \emptyset $ all $g\in N$ as $X$ is a geodesic space and the $\bigcup A_i$ are a Poincaré section for the flow on $X/\Gamma$, one obtains the bound    
\[   \G_r(\X_{\id})(x,\id) \gg  \sum_{g  \in N}   e^{ - d\left(g(\mathbf{o}),\mathbf{o}\right) \left( \delta(N) -  C^{-1}\log (\lambda r)   \right)},  \]
provided that $\lambda r \geq 1 $. Hence, $ \G_r(\X_{\id})(x,\id) = \infty $ for $\lambda r > 1$ as $\delta(N)$ is the Poincaré exponent of $N$. 
On the other hand, as $G$ is non-amenable, it follows from an application of a result by Zimmer in \cite{Zimmer:1978a} (see also \cite{Jaerisch:2013a}) that the product of $\mu$ on $\Sigma$ and the counting measure on $G$ is not conservative. Hence, Proposition 5.3 in \cite{Stadlbauer:2019} implies that $ \G_{\lambda^{-1}}(\X_{\id})(x,\id) < \infty$. Therefore, $R= 1/\lambda$.

We now  verify the reversibility condition. In order to do so, we make use of the generalisation by Adachi in \cite{Adachi:1986} of Rees' refinement (\cite{Rees:1981}) and obtain that we may in fact assume that there is an involution $\iota$ on the elements of the partition  which corresponds to the time reversal of the flow. This involution extends to finite words and, by a simple geometric argument, we have that $\kappa_w^{-1}  = \kappa_{\iota w}$ (see, e.g., the construction of the coding in \cite{DougallSharp:2015}). Hence,  
for $g \in \Gamma$, this implies that 
\begin{align*}
\G_R(\X_{gN})(x,\id)  & \asymp \sum_{w : \kappa_w \in g^{-1}N}  e^{ - \delta(N)d\left(\kappa_w(\mathbf{o}),\mathbf{o}\right)} = \sum_{w : \kappa_{\iota w} \in gN}   e^{ - \delta(N)d\left(\kappa_{\iota w}(\mathbf{o}),\mathbf{o}\right)} \asymp 
\G_R(\X_{\id})(x,gN) .
\end{align*}

\medskip
\noindent\textsc{Step 3. Identification of conformal measures.}  We now show that there is a canonical bijection between $\delta(N)$-conformal measures on $\partial X$ with respect to $N$ and $\varphi$-conformal measures on $\Sigma\times G$. 
Furthermore, as the definition of conformality in both cases precisely describes the behaviour along $G$ and $T$-orbits, respectively, the topological transitivity implies that a conformal measure is uniquely determined by its action on the $V_i$.

For a given $\delta(N)$-conformal measures on $\partial X$ with respect to $N$ the identity 
 \eqref{eq:N-conformal measure} holds by definition for $s= \delta(N)$, any $g \in N$ and each Borel set $A \subset \partial X$. Define $\widetilde{m}|_{V_i \times \{\id\}} := m$. In particular it follows for any finite word $w$ with $\kappa_w \in N$ and each Borel subset $(A,i) \subset w$ with $T^{|w|}(A,i) \subset V_j$  from \eqref{eq:geometric potential} that  
\begin{align*} \widetilde{m}\left(T^{|w|}((A,i) \times \{\id\})\right)  & =   \widetilde{m}\left((\kappa^{-1}_w(A),j)  \times \{\id\}\right) =  m\left(\kappa^{-1}_w(A)\right)\\
& =  \int_A \left((\kappa^{-1}_w)'\right)^{\delta(N)} dm = \int_{(A,i)} e^{\varphi(\eta,i)} d \widetilde{m}.
\end{align*} 
Hence, $ \widetilde{m}$ is conformal with respect to those branches of $T$ which start and end in $\Sigma \times \{\id\}$. Now assume that $[w,g]$ is a cylinder in $\Sigma \times G$. By topologically transitivity, there exists a cylinder $[v,\id]$ and $n \in \N$  such that $T^n([v,\id]) \supset  [w,g]$ and $T^n|_{[v,\id]}$ is injective. Then, as it easily can be verified (see, e.g., \cite{Stadlbauer:2019}),  
\[ d\widetilde{m}(A \cap [w,g]) := \int_{ [v,\id] \cap T^{-n}\left(A \cap [w,g]\right) } e^{-S_n(x)} d\widetilde{m}(x,\id) \]
extends $\widetilde{m}$ to a well-defined and conformal measure on $\Sigma \times G$. 

We now show the reverse direction. On order to do so, fix a conformal measure $m$ on $\Sigma \times G$ and set, for each $V_i$, $m_i := m|_{V_i \times \{\id\}}$. Moreover, assume that $g \in N$  and that $A$ is an open subset of $\partial X$ such that $A \subset V_i$ and $g(A) \subset V_j$ for some $1\leq i,j \leq n$. Hence, $A \subset V_i \cap g^{-1}(V_j)$ and, as the coding is topologically transitive, there exist $h \in \Gamma$ and  $1\leq k \leq n$ with 
\[ B:= h(V_k) \subset A, \quad h(U_k) \supset U_i \cup g^{-1} (U_j). \]
That is, there are open subset of $U_i \times B$ and $U_j \times g(B)$ which eventually flow into $h^{-1}(U_k) \times h^{-1}(V_k)$ and $(gh)^{-1}(U_k) \times (gh)^{-1}(V_k)$, respectively. Hence, there exist $s,t \in \N$ such that locally $T^s$ and $T^t$ are of the form 
\begin{align*}
T^s: & (B,i) \times \{\id_G \} \to V_k \times \{ h N \}, \; ((x,i), \{\id_G \}) \mapsto  ((h^{-1}(x),k), hN), \\
T^t: & (g(B),j) \times \{\id_G \} \to V_k \times \{ h N \}, \; ((x,j), \{\id\}_G) \mapsto  ((h^{-1}g^{-1} (x),k), hN),
\end{align*}
where we have used that  $ghN = hN$. 
Hence, by conformality of $m$ and the cocycle property of the Busemann function, it follows for any integrable function $f:B \to \R$ that
\begin{align}
\label{eq:conformality-without-g} \int f dm_i  & = \int \cL^s(\1_{(B,i) \times \{\id\}}f) dm = \int_{V_k \times \{hN\}} f(hx) e^{-\delta(N) B_{hx}  (h(\mathbf{o}), \mathbf{o}) }  d m,\\
\nonumber \int f\circ g^{-1}  dm_j & =  \int \cL^t(\1_{(gB,j) \times \{\id\}}f\circ g^{-1}) dm 
 =    \int_{V_k \times \{hN\}} f(hx) e^{-\delta(N) B_{ghx}  (gh(\mathbf{o}), \mathbf{o}) }  d m \\
\label{eq:conformality-with-g}  & =   \int_{V_k \times \{hN\}} f(hx) e^{-\delta(N) \left( B_{hx}  (h(\mathbf{o}), \mathbf{o}) + B_{ghx}  (g(\mathbf{o}),\mathbf{o}) \right) }  dm
.       
\end{align}
By combining \eqref{eq:conformality-without-g} with \eqref{eq:conformality-with-g}, one then obtains that 
$ dm_j(x) \circ g =  e^{- \delta(N)  B_{gx} (g(\mathbf{o}),\mathbf{o}) } dm_i(x)$ for $x \in B$. 
By transitivity of $T$ we may assume, without loss of generality, that $h \in N$ and $k=i$. Then, as $h(A) \subset B$
and $gh(A) = ghg^{-1}(gA) \subset gB$, we have that the restriction to $gB$ of some power of $\theta$ is given by $gh^{-1}g^{-1}$.  
Hence, for an integrable function 
${f}:A \to \R$, \eqref{eq:conformality-without-g} this representation of $m_j|_{gB}$ with respect to $m_i|_B$  
and
the argument in \eqref{eq:conformality-without-g} applied to $g(B)$ and $m_j$ imply that  
\begin{align*}
  \int  f  dm_i & =  \int_B f\circ h^{-1}  e^{\delta(N) B_{x}  (h(\mathbf{o}), \mathbf{o}) }  d m_i
   =  \int_B f\circ h^{-1}  e^{\delta(N) \left( B_{x}  (h(\mathbf{o}), \mathbf{o}) +  B_{gx}  (g(\mathbf{o}), \mathbf{o}) \right)}  d m_j\circ g\\
   & 
   =  \int_{g(B)} f\circ (gh)^{-1}   e^{\delta(N)  B_{x}  (gh(\mathbf{o}), \mathbf{o}) }  d m_j \\
   & 
   =  \int_{g(A)} f \circ g^{-1} e^{\delta(N) \left( B_{ghg^{-1}x}  (gh(\mathbf{o}), \mathbf{o}) -   B_{ghg^{-1}x}  (ghg^{-1}(\mathbf{o}), \mathbf{o}) \right) }
     d m_j \\
   & =   \int_{g(A)} f \circ g^{-1} e^{\delta(N)  B_{x}  (g^{-1}(\mathbf{o}), \mathbf{o}) } dm_j.
\end{align*}
In particular, if $g =\id$, this implies that $m_i|_{V_i \cap V_j} = m_j|_{V_i \cap V_j}$. Hence, $d{m}^\dagger(x) := dm_i(x) $ for $x \in V_i$ is a well defined measure on $\partial X$.   
Furthermore, by applying the above identity for arbitrary $g \in N$, \eqref{eq:geometric potential} shows that ${m}^\dagger$ is a $\delta(N)$-conformal measure for $N$ as defined in     
 \eqref{eq:N-conformal measure}.
 
Hence, we have shown that $m \mapsto (\tilde{m}(\Sigma\times\{\id\}))^{-1} \tilde{m}$ is a bijection from the set of  $\delta(N)$-conformal measures for $N$ on $\partial X$ to the set of $\varphi$-conformal measures on $\Sigma \times G$ which are normalised by giving measure 1 to $\Sigma\times\{\id\}$, and that the inverse of this map is given by $m \mapsto ({m}^\dagger(\partial X))^{-1} {m}^\dagger$.    

\medskip
\noindent\textsc{Step 4. Relating  $\widetilde\varphi$- and $\varphi$-conformal measures.}
Recall that the reference measure on $\Sigma$ is given by $d\nu = hd\mu$, where $hd\mu$ is the equilibrium state for the potential $\varphi$. As it is well known, $\nu$ is the unique conformal measure with respect to potential $\widetilde\varphi(\eta,i) = \varphi(\eta,i) + \log h(\eta,i)  - \log h(\theta(\eta,i))  - \log \lambda$, and the transfer operatores $L_\mu$ and $L_\nu$ of $\mu$ and $\nu$, respectively, are related through $ \lambda h L_\nu (f) = L_\mu(h f)$. 
It then follows immediately from the definitions, that this relation extends to $ \lambda h \cL_\mu (f) = \cL_\nu(h f)$
on the level of group extensions where we silently extended $h$ to a function on $\Sigma \times G$. Hence, as 
\[ \int f \cL_\mu^\ast (dm) = \int \cL_\mu(f) dm = \lambda \int \frac{\cL_\mu(h f/h)}{\lambda h} h dm = \lambda \int \frac{f}{h} \cL_\nu^\ast(hd m) \]
for any continuous function $f$ with compact support and each $\sigma$-finite measure $m$, one obtains that   $ dm \mapsto hdm$ defines a bijection between the space of conformal measures with respect to $\varphi$ and $\widetilde\varphi$, respectively.  
\end{proof}

We remark that Theorem \ref{theo:geometric-application} is related to conformal measures associated to ends of hyperbolic $n$-manifolds as introduced in \cite{AndersonFalkTukia:2007}. In there, the authors construct for an arbitrary hyperbolic $n$-manifold a finite family of open sets such that each  $\alpha$-conformal measure can be represented as a sum of conformal measures, where each of this measures is associated to one of these open sets. Hence, our result in here might be seen as a refinement of the above for regular covers as we obtain a complete description of the set of $\delta(N)$-conformal measures. In particular, the above  shows that these ends could be replaced through an iterative construction by $\partial G$, provided that $\Gamma/N$ is word hyperbolic.    

Moreover, Shwartz recently obtained a similar result for regular covers of cocompact Fuchsian groups and with respect to $\alpha$-conformal measures  with $\alpha > \delta(N)$ (see \cite{Shwartz:2019}). This restriction is a consequence of the version of Shwartz of Ancona's inequality which does not allow to include the critical parameter.     
On the other hand, by applying Theorem 4.1 in  \cite{Shwartz:2019a} to our setting, we obtain as by Shwartz in \cite{Shwartz:2019} an ergodic theoretic description of these minimal measures as a corollary. In order to do so, we recall the notions of limit sets and uniform approximating sequences. The limit set of a group of isometries $\Gamma$ is defined by $\Lambda(\Gamma):= \overline{\Gamma(\{\mathbf{o}\})} \cap \partial X$, that is $\Lambda(\Gamma)$ is the set of accumulation points of the orbit $\Gamma(\{\mathbf{o}\})$ in $\partial X$. Moreover, we say that the sequence  $(g_n)$ in $\Gamma$ uniformly approximates $\eta \in \partial X$ if $\lim_{n \to \infty} g_n(\mathbf{o}) = \eta$ and there exists $C>0$, depending on $(g_n)$, such that the distance between the geodesic ray from $\mathbf{o}$ to $\eta$ is bounded from above by $C$.

\begin{corollary} Under the assumptions of Theorem \ref{theo:geometric-application}, the following holds. 
\begin{enumerate}
\item Assume that $\mu$ is a minimal, $\delta(N)$-conformal measure for $N$. Then $\mu$ is ergodic for the action of $N$ on $\partial X$ and there exists $\sigma \in \partial(\Gamma/N)$ such that for $\mu$-a.e. $\eta \in \Lambda(N)$, and for every $(g_n)$ in $\Gamma$ which uniformly approximates $\eta$, we have that  $\lim_{n \to \infty} g_nN = \sigma$. 
\item Assume that $\sigma \in \partial(\Gamma/N)$. Then there exists a unique $\delta(N)$-conformal measure $\mu$ for $N$ such that for $\mu$-a.e. $\eta \in \Lambda(N)$, and for every $(g_n)$ in $\Gamma$ which uniformly approximates $\eta$, we have that  $\lim_{n \to \infty} g_nN = \sigma$. 
\end{enumerate}
\end{corollary}

\begin{proof} We begin with the proof of the first part. The ergodicity of  $\mu$ is an immediate consequence of minimality as any $G$-invariant set  $A \subset \partial X$ defines a $\delta(N)$-conformal measure $dm:= \mathbf{1}_A d\mu$ with $m \leq \mu$. Hence, $\mu(A) = 0 $ or $\mu(A) = 1$. 

In order to show convergence, we employ Theorems \ref{theo:geometric-boundary} and \ref{theo:geometric-application} as they imply that there exists $\sigma \in \partial(\Gamma/N)$ such that $\mu = \mu_\sigma$. By 
Theorem 4.1 in  \cite{Shwartz:2019a} and the coding constructed in the proof of the theorem, it then follows  for almost every element in $((\eta,i),\id) \in  V_i \times \{\id\}$ that the orbit $(T^n((\eta,i),\id))$ converges to the element $\Xi^{-1}(\sigma)$ in the Martin boundary (cf. Theorem \ref{theo:continuidade} for the definition of $\Xi$). However, by Theorem \ref{theo:geometric-boundary} above, this implies that the second coordinate of $(T^n((\eta,i),\id))$ converges to $\sigma \in  \partial(\Gamma/N)$. It hence remains to relate this convergence with uniform approximation.   

In order to do so, set $g_n :=\kappa_{\theta^{n-1}(\eta,i)}\circ \cdots \circ \kappa_{(\eta,i)}$ and choose an element $\xi \in U_i$. Then, by the coding construction, the geodesic $(\xi,\eta)$ from $\xi$ to $\eta$ passes through the closure of $g_n(\bigcup_j R_j)$, where the $R_j$ refer to the atoms of the coding construction in the proof above. However, as $\Gamma$ is convex-cocompact, the diameter of the projection of $\bigcup_j R_j$ to $X$ is finite. Therefore, $(g_n(\mathbf{o}))$ stays within a bounded distance from $(\xi,\eta)$ and converges to $\eta$. By combining this observation with the above convergence, one then obtains that $(g_n(\mathbf{o}))$ uniformly approximates $\eta$ and that $g_nN \to \sigma$ almost surely.

It is left to prove the claim for an arbitrary sequence $(h_n)$ which uniformly approximates $\eta$. As the return time to $\bigcup_j R_j$ is bounded from above, it follows that $\sup_n d_X(g_n(\mathbf{o}),g_{n+1}(\mathbf{o})) < \infty$. In particular, there exists a sequence $(n_k)$ such that $\sup_k d_X(h_k(\mathbf{o}),g_{n_k}(\mathbf{o})) < \infty$. Hence, as  $\Gamma$ acts discontinuously on $X$, the set $\{ h_k^{-1}g_{n_k} : k \in \N\}$ is finite. Therefore,  $h_kN$ and $g_{n_k}N$ stay within a bounded distance with respect to the theorem.

The second part of the theorem is consequence of Part (i) combined with the fact that there is a bijection between $\partial(\Gamma/N)$ and the set of minimal, $\delta(N)$-conformal measures.
\end{proof}

\section{Appendix: Reduced measures and the domination principle}\label{sec:appendix}
In this part, following the exposition in \cite{Woess:2009}, well-known ideias from potential theory for Markov operators are adapted to the setting of Ruelle operators on locally compact shift spaces (see also \cite{Shwartz:2019a}).
We begin with the following version of the Riesz decomposition theorem. Throughout this part, we assume that the potential is transient.
\begin{proposition}[Lemma 3.2 in  \cite{Shwartz:2019a}]
Let $\mu$ be a $1/r$-excessive measure for $0< r \leq R$. Then there exists a unique pair of Radon measures
$\mu_0$ and $\nu$ such that $\mu_0$ is $1/r$-conformal and $\mu = \mu_0 + \G_{r}^\ast(\nu)$ in the sense that
\[ \int f d\mu  =   \int f d\mu_0 + \int  \G_{r}(f)(z) d\nu(z)  \]
for any continuous $f$  with compact support. Moreover, $\nu = \mu - r^{-1} \cL^\ast(\mu)$.
\end{proposition}
\begin{proof} The existence follows from Lemma 3.2 in  \cite{Shwartz:2019a}. Now assume that $\mu_0 + \G_{r}^\ast(\nu) = \tilde{\mu}_0 + \G_{r}^\ast(\tilde{\nu})$. By applying
 $r^{-1}\cL^\ast$ to both sides, it follows that
 \[  \mu_0 + \G_{r}^\ast(\nu) - \nu = \tilde{\mu}_0 + \G_{r}^\ast(\tilde{\nu}) - \tilde{\nu}.\]
 Hence, $\nu = \tilde{\nu}$.
\end{proof}
Recall that $\mu \leq \nu$ if $\mu(A) \leq \nu(A)$ for all $A \in \mathcal{B}$. The
Riesz decomposition theorem has the following useful consequences.
\begin{lemma} \label{lem:dominated by a potential} Assume that $\mu$ is $1/r$-excessive and that $\mu \leq \G_{r}^\ast(\nu)$ for a measure $\nu$ such that $\G_{r}^\ast(\nu)$ is $\sigma$-finite. Then there exists $\nu_0$ such that $\mu = \G_{r}^\ast(\nu_0)$. In particular, if $\mu$ is harmonic, then $\mu=0$.
\end{lemma}

\begin{proof}  By the above, $\mu = \mu_0 + \G_{r}^\ast(\nu_0)$. Therefore,
$\mu_0 \leq \G_{r}^\ast(\nu)$ and  $\mu_0=0$ since
\[\mu_0 = r^{-n}(\cL^n)^\ast(\mu_0) \leq r^{-n}(\cL^n)^\ast(\G_{r}^\ast(\nu)) \to 0.\]
The assertion then follows from the uniqueness of the Riesz decomposition.
\end{proof}

For a family of $\sigma$-finite measures $\{\mu_i : i \in I\}$, define $\mu^\dagger(A):= \inf_{i\in I} \mu_i(A)$ and
\[ \bigwedge_{i\in I} \mu_i(A) := \inf\left\{ \sum_{j=1}^\infty \mu^\dagger(B_j): \bigcup_{j=1}^\infty B_j = A, B_j \in \mathcal{B} \right\}, \quad \hbox{ for } A\in \mathcal{B}.\]
We then refer to $\bigwedge_{i\in I} \mu_i(A)$ as the infimum of the family $\{\mu_i : i \in I\}$.
\begin{proposition} The infimum of a family of Radon measures is a Radon measure. Moreover, the infimum of a family of $\lambda$-excessive measures is $\lambda$-excessive.
\end{proposition}
\begin{proof} We begin showing that $\mu :=  \bigwedge_{i\in I} \mu_i$ is a measure. In order to do so, first observe that for a partition of $A \in \mathcal{B}$ into a $\{A_i \in \mathcal{B}: i \in \N\}$ that  $\mu(A) \leq  \sum_i \mu(A_i)$ by construction. On the other hand, observe that for partitions $\{ B_j : j \in \N \}$,  $\{ B^\ast_k : k \in \N \}$ of $A$ into Borel sets such that the second is finer than the first,
 \begin{equation}\label{eq:refinement}  \sum_{j =1}^\infty \mu^\dagger(B_j) \leq \sum_{k =1}^\infty \mu^\dagger(B^\ast_k).\end{equation}
In particular, this implies that $\mu(A) \geq  \sum_i \mu(A_i)$, as each partition of $A$ has a refinement which is measurable with respect to $\sigma(\{A_i: i \in \N\})$.  Hence, $\mu$ is $\sigma$-additive and therefore a measure. Moreover, as $\mu(A) \leq \mu_i(A)$ for all $i$, $\mu$ is $\sigma$-finite. The Radon property follows immediately from
 $\mu(A) \leq \mu_i(A)$ for all $i \in I$ and $A \in \mathcal{B}$, which proves the first assertion.

Now assume that $f:X \to [0,\infty)$ is uniformly continuous and bounded, that $A \in \mathcal{B}$ with $\mu(A) < \infty $ and that  $\epsilon$ is arbitrary. By applying \eqref{eq:refinement}, we may suppose that $\mu(A) \leq \sum_j \mu^\dagger(B_j) \leq \mu(A)+\epsilon$, where the $\{ B_j : j \in \N \}$ is a partition of $A$ into sets of diameter $\delta$. If $\delta$ is chosen sufficently small, uniform continuity implies that
 \begin{align*}\int_A f d\mu &  =  \sum_{j =1}^\infty \int_{B_j} f d\mu  \leq  \sum_{j =1}^\infty   \left(\inf_{x \in B_j} f(x) + \epsilon  \right) \mu^\dagger(B_j) \leq \sum_{j =1}^\infty  \int_{B_j} f d\mu_i + \epsilon ( \mu(A)+\epsilon) \\
 & \leq  \inf_{i\in I} \int_A f  d\mu_i + \epsilon \mu(A) + \epsilon^2.
 \end{align*}
As $\epsilon$ is arbitrary, it follows that $\int_A f d\mu \leq   \inf_i \int_A f  d\mu_i $. The remaining assertion easily follows from this.
\end{proof}

Now assume that $\mu$ is a   $\lambda$-excessive measure and that $A \in \mathcal{B}$. Then we refer to
\[ R_A(\mu) := \bigwedge \left\{\nu  \in \mathcal{R}: \cL^\ast(\nu) \leq \lambda \nu, \nu|_A \geq \mu|_A  \right\}\]
as the \textit{reduced measure} associated with $\mu$ on $A$, which is a
is a well-defined, $\lambda$-excessive Radon measure by the above Proposition. Furthermore, if $A = \Sigma \times K$, for some finite $K \subset G$, then $\G_{r}^\ast(\mu|_A)$ is well-defined, $\lambda$-excessive and $\G_{r}^\ast(\mu|_A) \geq \mu|_A$. In particular, $R_A(\mu) \leq \G_{r}^\ast(\mu|_A)$.  Hence, Lemma \ref{lem:dominated by a potential} implies that there exists $\nu_0$ such that $R_A(\mu) = \G_{r}^\ast(\nu_0)$ and, as $\nu_0 = R_A(\mu)- r^{-1} \cL^\ast(R_A(\mu))$, $\nu_0$ is a Radon measure.
Also note that $R_A(\mu)|_A = \mu|_A$ by construction. Hence, by letting $K \to G$, one immediately obtains the following.
\begin{proposition}\label{prop:approximation-by-potentials}  Assume that $\mu$ is  $1/r$-excessive. Then there exists an increasing sequence of Radon measures $(\nu_n)$ such that
$\mu(A) = \lim_{n\to \infty} \G_{r}^\ast(\nu_n) (A)$, for all $A \in \mathcal{B}$.
\end{proposition}

Now assume that $A$ is measurable with respect to the partition into $n$-cylinders, for some $n \in \N$.
Then $\1_A$ is Lipschitz continuous, and, in particular, $\cL_A(f) := \cL(\1_A f)$ acts on continuous functions with compact support. Furthermore, consider
\[ \G_r^A := \1_A \sum_{n=0}^\infty r^n(\cL_A)^n, \quad \mathcal{F}_A:=  \1_A  \sum_{n=0}^\infty r^{n} (\cL_{A^c})^n .\]
Observe that this choice of $A$ implies by Proposition \ref{prop:finiteness of G_R} that $\G_r^A$ and $\mathcal{F}_A$ act on Lipschitz functions with support on $\Sigma \times K$, for $K \subset G$ finite, for $0 < r \leq R$. In particular, the actions of $(\G_r^A)^\ast$ and $\mathcal{F}_A^\ast$ on $\sigma$-finite measures are well defined.
\begin{theorem}\label{theo:reduced_measure} Assume that $\mu$ is $1/r$-excessive and that $A$ is measurable with respect to the partition into $n$-cylinders, for some $0 < r \leq R$ and $n \in \N$. Then the reduced measure on $A$ is equal to $\mathcal{F}_A^\ast(\mu)$.
\end{theorem}

\begin{proof} Set $B:= A^c$. By Proposition  \ref{prop:approximation-by-potentials}, there exists a monotone sequence of Radon measures $(\nu_n)$ such that $\G_{r}^\ast(\nu_n) \to \mu$. Moreover, by decomposing orbits with respect to the first entry to $A$, one obtains that $\G_r =\G_r^{B}  + \G_r \circ \mathcal{F}_A$, which then implies that
\[ \mathcal{F}_A^\ast \circ \G_{r}^\ast(\nu_n)  = (\G_r - \G_r^{B})^\ast (\nu_n) \leq  \G_r^\ast (\nu_n).  \]
By taking the limit as $n \to \infty$, one obtains that $\mathcal{F}_A^\ast(\mu)([w,g]) \leq  \mu([w,g])$ for any cylinder set, and therefore $\mathcal{F}_A^\ast(\mu)\leq \mu$. As, by construction, $\mathcal{F}_A^\ast(\mu)|_A = \mu|_A$, it follows that $\mu|_A \leq \nu|_A$ implies that
$\mathcal{F}_A^\ast(\mu) \leq \mathcal{F}_A^\ast(\nu)$ globally. Hence, if $\mathcal{F}_A^\ast(\mu)$ is excessive, then $\mathcal{F}_A^\ast(\mu) = R_A(\mu)$.

It remains to show that $\mathcal{F}_A^\ast(\mu)$ is excessive. By iterated application of $\cL^\ast(\mu)\leq r^{-1} \mu$ one obtains  for a test function $h\geq 0$ that
\begin{align*}
\int \1_A h d\mu & \geq r \int \cL(\1_A h) d\mu  =  r \int \1_A \cL_A(h) d\mu + r \int \1_B \cL_A(h) d\mu\\
& \geq r \int \1_A \cL_A(h) d\mu + r^2 \int  \1_A \cL_B\cL_A(h) d\mu + r^2 \int  \1_B \cL_B\cL_A(h) d\mu \\
& \geq \int \1_A \sum_{k=0}^n r^{k+1} \cL_B^k\cL_A(h) d\mu + r^{n+1} \int  \1_B \cL_B^n\cL_A(h) d\mu.
\end{align*}
Hence, by monotone convergence,
\[\int \1_A h d\mu  \geq  \int \1_A \sum_{n=0}^\infty r^{n+1} \cL_B^n\cL_A(h) d\mu.\]
On the other hand,
\begin{align*}
 r \mathcal{F}_A\circ \cL & = \1_A\sum_{n=0}^\infty r^{n+1} \cL_{B}^n\circ (\cL_B + \cL_A) = \mathcal{F}_A - \1_A +   \1_A\sum_{n=0}^\infty r^{n+1} \cL_{B}^n\circ \cL_A,
\end{align*}
which implies that
\begin{align*}
r\int \cL h d\mathcal{F}_A^\ast(\mu) & = \int  h d\mathcal{F}_A^\ast(\mu) -  \int_A h - \sum_{n=0}^\infty r^{n+1} \cL_B^n\cL_A(h) d\mu \leq \int  h d\mathcal{F}_A^\ast(\mu),
\end{align*}
proving that $\mathcal{F}_A^\ast(\mu)$ is $1/r$-excessive.
\end{proof}

As $\mathcal{F}_A^\ast(\mu)$ is dominated by $G_r^\ast(\mu|_A)$, it follows from Lemma \ref{lem:dominated by a potential} that there exists a unique $\nu$ with $\mathcal{F}_A^\ast(\mu) = \G_r^\ast(\nu_A)$. If $\mu = \G_r^\ast(\nu)$ for some $\nu$, this gives rise to a map $\nu \mapsto \nu_A$, where $\nu_A$ is referred to as the \emph{balayée} of $\nu$ and can be constructed explicitly as follows. Set
\[\mathcal{R}_A (f):=
\sum_{n=0}^\infty r^{n} ( \1_{A^c}\cL)^n(\1_A f).
\]
As each orbit with at least one visit to $A$ can be decomposed either with respect to the last or the first visit to $A$, it follows that $\G_r \circ \mathcal{F}_A = \mathcal{R}_A \circ \G_r$. Therefore, the balayée  of $\nu$ is given by $ \nu_A = \mathcal{R}_A^\ast(\nu)$ as, for $\mu = \G_r^\ast (\nu)$,
\[ R_A(\mu) = \mathcal{F}_A^\ast (\mu) =  \mathcal{F}_A^\ast\circ \G_r^\ast (\nu) =   \G_r^\ast \circ \mathcal{R}_A^\ast (\nu) = \G_r^\ast(\nu_A).
\]
In particular, if $\nu$ is supported on $A$, then  $\nu_A = \mathcal{R}_A^\ast (\nu) = \nu$. This proves the following result, known as domination principle.
\begin{theorem}\label{theo:domination} Assume that $A$ is measurable with respect to the partition into $n$-cylinders for some $n \in \N$ and that $\nu$ is a finite Radon measure whose support is contained in $A$ such that $ \G_r^\ast(\nu)$ is well defined ($0< r \leq R$).
If $\mu$ is $1/r$-excessive and $\mu|_A \geq  \G_r^\ast(\nu)|_A$, then $\mu \geq  \G_r^\ast(\nu)$.
\end{theorem}

\section*{Acknowledgments} The authors would like express their gratitude for support of the Post-Graduate program of the Universidade Federal da Bahia, where the first author obtained her PhD. Her thesis, defended in November 2019, comprises the main part of the results of this article and was supervised by the second author. In particular, the authors acknowledge support by  CAPES and CNPq: The first author was supported by CAPES during her PhD, and the second was partially supported  by CAPES (Programa PROEX da Pós-Graduação em Matemática do IM-UFRJ) e CNPq (PQ 312632/2018-5, Universal 426814/2016-9).


\end{document}